\documentclass[a4paper,10pt]{amsart}
\makeatletter
\newcommand{\rmnum}[1]{\romannumeral #1}
\newcommand{\Rmnum}[1]{\expandafter\@slowromancap\romannumeral #1@}
\makeatother
\usepackage{hyperref}
\usepackage{mathrsfs}
\usepackage{amsfonts}
\usepackage{txfonts}
\usepackage{amssymb}
\usepackage[arrow,matrix]{xy}
\usepackage{amsmath,amssymb,amscd,bbm,amsthm,mathrsfs,dsfont}
\usepackage{amsmath,amscd}\input amssym.def
\usepackage{amsfonts,amssymb}
\usepackage{amsmath,amscd}\input amssym.def
\newtheorem{theorem}{Theorem}[section]
\newtheorem{lemma}[theorem]{Lemma}
\theoremstyle{definition}
\newtheorem{defn}[theorem]{Definition}
\newtheorem{exa}[theorem]{Example}

\newtheorem{corollary}[theorem]{Corollary}
\newtheorem{prop}[theorem]{Proposition}
\theoremstyle{remark}
\newtheorem{remark}[theorem]{Remark}
\numberwithin{equation}{section}

\begin{document}

\title{PBW-basis for universal enveloping algebras of differential graded Poisson algebras}

\author{XianGuo Hu}
\address{Department of Mathematics, Zhejiang Normal University, Jinhua, Zhejiang, 321004 P.R. China}
\email{2015210420@zjnu.edu.cn}

\author{Jiafeng L\"{u}$^*$}
\address{Department of Mathematics, Zhejiang Normal University, Jinhua, Zhejiang, 321004 P.R. China}
\email{jiafenglv@zjnu.edu.cn}

\author{Xingting Wang}
\address{Wang: Department of Mathematics, Temple University, Philadelphia, 19122, USA }
\email{xingting@temple.edu}

\thanks{This work was supported by National Natural Science Foundation of China (Grant No.11571316) and Natural
Science Foundation of Zhejiang Province (Grant No. LY16A010003).}
\thanks{*corresponding author}


\begin{abstract}
For any differential graded (DG for short) Poisson algebra $A$ given by generators
and relations, we give a ``formula" for computing the universal
enveloping algebra $A^e$ of $A$. Moreover, we prove that $A^e$ has a
Poincar\'{e}-Birkhoff-Witt basis provided that $A$ is a graded commutative polynomial
algebra. As an application of the PBW-basis, we show that a DG symplectic ideal of a DG Poisson algebra $A$ is the annihilator of a simple DG Poisson $A$-module, where $A$ is the DG Poisson homomorphic image of a DG Poisson algebra $R$ whose underlying algebra structure is a graded commutative polynomial algebra.
\end{abstract}
\maketitle

\medskip
\section{Introduction}
The notion of Poisson algebras arises naturally in the study of Hamiltonian mechanics and Poisson
geometry. Recently, many important generalizations on Poisson algebras have been obtained in both
commutative and noncommutative settings: Poisson orders
\cite{BG}, Poisson PI algebras \cite{MPR}, graded Poisson algebras
\cite{CFL}, double Poisson algebras \cite{V}, Novikov-Poisson
algebras \cite{Xu}, Quiver Poisson algebras \cite{YYZ},
noncommutative Leibniz-Poisson algebras \cite{CD}, left-right
noncommutative Poisson algebras \cite{CDL}, noncommutative Poisson
algebras \cite{X} and differential graded Poisson algebras
\cite{LWZ}, etc. An interesting and practical idea to develop Poisson algebras is to study Poisson universal
enveloping algebras, which was first introduced by Oh in 1999
\cite{O}, and later Oh, Park and Shin studied the
Poincar\'{e}-Birkhoff-Witt basis (PBW-basis for short) for Poisson
universal enveloping algebras \cite{OPS}. Since then, Poisson universal enveloping algebras have been studied in a series of papers \cite{U, YYY}. In particular, the second author and the third author of the present paper studied the universal enveloping algebras of Poisson Hopf algebras and Poisson Ore-extensions \cite{LWZ1, LWZ2}. Our main aim of this paper is
to study the PBW-basis for universal enveloping algebras of
DG Poisson algebras.

In \cite{HL, HLe, J, Va}, Hodges, Levasseur, Joseph and Vancliff proved that symplectic leaves of certain Poisson varieties
correspond bijectively to primitive ideals of the respective quantum algebras. After that, Oh defined the symplectic ideal of a
Poisson algebra and proved that there is a one to one
correspondence between the primitive ideals of quantum $2\times 2$
matrix algebra and the symplectic ideals of a Poisson algebra
constructed appropriately \cite{Oh}. Moreover, he and two other authors showed that the
symplectic ideal of a Poisson homomorphic image of a Poisson polynomial algebra
$\Bbbk[x_1, ..., x_n]$ is the annihilator of a simple Poisson
module. Motivated by the DG version for symplectic ideals, there is a natural question: Is a DG symplectic ideal of a DG Poisson algebra $A$ the annihilator of a simple DG Poisson $A$-module? The present paper gives a positive answer.

The paper is organized as follows. In Section 2, we briefly review some basic concepts related to DG Poisson algebras, DG Poisson modules and universal enveloping algebras of DG Poisson algebras. In particular, we construct the universal enveloping algebra of any DG Poisson algebra given by generators and relations. Section 3 is devoted to the study of PBW-basis for universal enveloping algebras. To be specific, for a DG Poisson algebra $R$ whose underlying algebra structure is a graded commutative polynomial algebra, by using Gr\"{o}bner-Shirshov basis theory developed in \cite{KL} and \cite{KLe}, we prove that the universal enveloping algebra $R^e$ has a PBW-basis, which is analogous to the PBW-basis for universal enveloping algebras of Lie algebras. In the last section, we focus on the simple DG Poisson module. As an application of the PBW-basis theorem for universal enveloping algebras of DG Poisson algebras, we prove that a DG symplectic ideal of a DG Poisson algebra $A$ is the annihilator of a simple DG Poisson $A$-module, where $A$ is the DG Poisson homomorphic image of $R$.

Throughout the whole paper, $\mathbb{Z}$ denotes the set of
integers, $\Bbbk$ denotes a base field and everything is over
$\Bbbk$ unless otherwise stated, all (graded) algebras are assumed
to have an identity and all (graded) modules are assumed to be
unitary.

\medskip
\section{Universal enveloping algebras of differential graded Poisson algebras}
In this section, first we briefly review some basic definitions and
properties of DG Poisson algebras and universal enveloping algebras, then we construct the universal
enveloping algebra of any DG Poisson algebra $A$ given by generators
and relations.

\subsection{DG Poisson algebras} By a graded algebra we mean
a $\mathbb{Z}$-graded algebra. A DG algebra is a graded algebra with
a $\Bbbk$-linear homogeneous map $d: A\rightarrow A$ of degree 1,
which is also a graded derivation. Any graded algebra can be viewed
as a DG algebra with differential $d=0$; in this case it is called a
DG algebra with trivial differential. Let $A$, $B$ be two DG
algebras and $f: A\rightarrow B$ be a graded algebra map of degree
zero. Then $f$ is called a DG algebra map if $f$ commutes with the
differentials.

\begin{defn}
Let $A$ be a graded $\Bbbk$-vector space. If there is a
$\Bbbk$-linear map
\begin{center}
$\{\cdot, \cdot\}: A\otimes A\rightarrow A$
\end{center}
of degree 0 such that:

\begin{enumerate}
\item[(\rmnum{1})] (graded antisymmetry): $\{a, b\}=-(-1)^{|a||b|}\{b, a\}$;

\item[(\rmnum{2})] (graded Jacobi identity): $\{a, \{b, c\}\}=\{\{a, b\},
c\}+(-1)^{|a||b|}\{b, \{a, c\}\}$,
\end{enumerate}
for any homogeneous elements $a, b, c\in A$, then $(A, \{\cdot, \cdot\})$ is
called a graded Lie algebra.
\end{defn}

\begin{defn} \cite{LWZ}
Let $(A, \cdot)$ be a graded $\Bbbk$-algebra. If there is a
$\Bbbk$-linear map
\begin{center}
$\{\cdot, \cdot\}: A\otimes A\rightarrow A$
\end{center}
of degree 0 such that

\begin{enumerate}
\item[(\rmnum{1})] $(A, \{\cdot, \cdot\})$ is a graded Lie algebra;

\item[(\rmnum{2})] (graded commutativity): $a\cdot b=(-1)^{|a||b|}b\cdot a$;

\item[(\rmnum{3})] (biderivation property): $\{a, b\cdot c\}=\{a, b\}\cdot
c+(-1)^{|a||b|}b\cdot \{a, c\}$,
\end{enumerate}
for any homogeneous elements $a, b, c\in A$, then $A$ is called a
graded Poisson algebra. If in addition, there is a $\Bbbk$-linear
homogeneous map $d: A\rightarrow A$ of degree 1 such that $d^2=0$ and

\begin{enumerate}
\item[(\rmnum{4})] (graded Leibniz rule for bracket): $d(\{a, b\})=\{d(a),
b\}+(-1)^{|a|}\{a, d(b)\}$;

\item[(\rmnum{5})] (graded Leibniz rule for product): $d(a\cdot b)=d(a)\cdot
b+(-1)^{|a|}a\cdot d(b)$,
\end{enumerate}
for any homogeneous elements $a, b\in A$, then $A$ is called a DG
Poisson algebra, which is usually denoted by $(A, \cdot, \{\cdot, \cdot\},
d)$, or simply by $(A, \{\cdot, \cdot\}, d)$ or $A$ if no confusions arise.
\end{defn}

\begin{remark}
For a DG algebra $B$, assume throughout the paper that $B_P$ is the
DG Poisson algebra $B$ with the ``standard graded Lie bracket": $[a,
b]=ab-(-1)^{|a||b|}ba$ for any homogeneous elements $a, b\in B$.
\end{remark}

Note that an ideal $I$ of DG Poisson algebra $A$ is called a DG
Poisson ideal if $d(I)\subseteq I$, $\{I, A\}\subseteq I$. Let $B$
be another DG Poisson algebra. A graded algebra map
$\rho: A\rightarrow B$ is said to be a DG Poisson algebra map if
$\rho\circ d_A=d_B\circ \rho$ and $\rho(\{a, b\}_A)=\{\rho(a),
\rho(b)\}_B$ for all homogeneous elements $a, b\in A$. For example,
if $I$ is a DG Poisson ideal of $A$, then the canonical projection
$\pi_{I}: A\rightarrow A/I$ is a DG Poisson algebra map.

We denote by \textbf{DG(P)A} the category of DG (Poisson) algebras
whose morphism space consists of DG (Poisson) algebras map.

Now, we recall the definition of DG Poisson modules over DG Poisson algebras.
\begin{defn} \cite{LWZ}\label{defn1}
Let $A=(A, \cdot, \{\cdot, \cdot\}, d)\in \textbf{DGPA}$, and $M$ be a left
graded module over $A$. We call $M$ a left DG Poisson module over
$A$ provided that
\begin{enumerate}
\item[(\rmnum{1})] $(M, \partial)$ is a left DG module over the DG algebra $A$. That is to say, there is a $\Bbbk$-linear map $\partial: M\rightarrow M$ of degree 1 such that $\partial^2=0$ and
\begin{center}
$\partial(am)=d(a)m+(-1)^{|a|}a\partial(m)$
\end{center}
for all homogeneous elements $a\in A$ and $m\in M$. Here $\partial$ is also called the differential of $M$.

\item[(\rmnum{2})] $M$ is a left graded Poisson module over the graded Poisson algebra $A$. That is to say, there is a $\Bbbk$-linear map $\{\cdot, \cdot\}_M: A\otimes M\rightarrow M$ of degree 0 such that
\begin{enumerate}
  \item[(\rmnum{1}a)] $\{a, bm\}_M=\{a, b\}_Am+(-1)^{|a||b|}b\{a, m\}_M$;
  \item[(\rmnum{1}b)] $\{ab, m\}_M=a\{b, m\}_M+(-1)^{|a||b|}b\{a, m\}_M$, and
  \item[(\rmnum{1}c)] $\{a, \{b, m\}_M\}_M=\{\{a, b\}_A, m\}_M+(-1)^{|a||b|}\{b, \{a, m\}_M\}_M$,
\end{enumerate}
for all homogeneous elements $a, b\in A$ and $m\in M$.

\item[(\rmnum{3})] the $\Bbbk$-linear map $\partial$ is compatible with the bracket $\{\cdot, \cdot\}_M$. That is, we have
\begin{center}
$\partial(\{a, m\}_M)=\{d(a), M\}_M+(-1)^{|a|}\{a, \partial(m)\}_M$,
\end{center}
for all homogeneous elements $a\in A$ and $m\in M$.
\end{enumerate}
Similarly, a left DG Poisson module $M$ over a DG Poisson algebra $A$ is usually denoted by $(M, \{\cdot, \cdot\}_M, \partial)$, or simply by $M$ if there are no confusions.
\end{defn}
\begin{defn}
Let $A=(A, \cdot, \{\cdot, \cdot\}, d)\in \textbf{DGPA}$, and $(M, \{\cdot, \cdot\}_M,
\partial)$ be a left DG Poisson module over $A$. A left graded
submodule $N\leq M$ is called a left DG Poisson submodule provided
that $\partial(N)\subseteq N$ and $\{A, N\}_M\subseteq N$, which is
usually denoted by $N\leq_pM$.
\end{defn}

\subsection{Universal enveloping algebras of DG Poisson algebras}
In this subsection, we recall the definition and some properties of
the universal enveloping algebra of a DG Poisson algebra $A=(A, \cdot,
\{\cdot, \cdot\}, d)$.
\begin{defn} \cite{LWZ}
Let $A=(A, \cdot, \{\cdot, \cdot\}, d)\in \textbf{DGPA}$ and $(A^{ue},
\partial)\in \textbf{DGA}$. We call $(A^{ue}, \partial)$ is a
universal enveloping algebra of $A$ if there exist a DG algebra map
$\alpha: (A, d)\rightarrow (A^{ue}, \partial)$ and a DG Lie algebra
map $\beta: (A, \{\cdot, \cdot\}, d)\rightarrow (A_P^{ue}, [\cdot, \cdot],
\partial)$ satisfying
\begin{align*}
\alpha(\{a, b\})&=\beta(a)\alpha(b)-(-1)^{|a||b|}\alpha(b)\beta(a),\\
\beta(ab)&=\alpha(a)\beta(b)+(-1)^{|a||b|}\alpha(b)\beta(a),
\end{align*}
for any homogeneous elements $a, b\in A$, such that for any $(D,
\delta)\in \textbf{DGA}$ with a DG algebra map $f: (A, d)\rightarrow
(D, \delta)$ and a DG Lie algebra map $g: (A, \{\cdot, \cdot\}, d)\rightarrow
(D_P, [\cdot, \cdot], \delta)$ satisfying
\begin{align*}
f(\{a, b\})&=g(a)f(b)-(-1)^{|a||b|}f(b)g(a),\\
g(ab)&=f(a)g(b)+(-1)^{|a||b|}f(b)g(a),
\end{align*}
for all $a, b\in A$, then there exists a unique DG algebra map
$\phi: (A^{ue}, \partial)\rightarrow (D, \delta)$, making the diagram
\begin{eqnarray*}
\xymatrix {
 (A, \{\cdot, \cdot\}, d) \ar[rr]^{\alpha, ~\beta} \ar[dr]_{f, ~g}
                &  &    (A^{ue}, \partial) \ar[dl]^{\exists !\phi}    \\
                &   (D, \delta)               }
\end{eqnarray*}
``bi-commute", i.e., $\phi\alpha=f$ and $\phi\beta=g$.
\end{defn}
For an $A\in \textbf{DGPA}$, we denote by $(A^{ue}, \alpha, \beta)$
the universal enveloping algebra of $A$. Note that universal
enveloping algebra $A^{ue}$ of a DG Poisson algebra $A$ was defined
in order that a $\Bbbk$-vector space $M$ is a DG Poisson $A$-module
if and only if $M$ is a DG $A^{ue}$-module and that universal
enveloping algebra of $A$ is unique up to isomorphic (see \cite{LWZ}).

\begin{prop}
Let $(A^{ue}, \alpha, \beta)$ be the universal enveloping algebra of
a finitely generated DG Poisson algebra $A$. Then $\alpha$ is
injective.
\end{prop}

\begin{proof}
For every $a\in A$, define $\gamma(a)$, $\delta(a)\in
End_{\Bbbk}(A)$ by
\begin{center}
$\gamma(a)(b)=ab, \delta(a)(b)=\{a, b\}$,
\end{center}
for all $b\in A$. As we all know, $End_{\Bbbk}(A)$ is a graded
endomorphism ring if $A$ is a finitely generated $\Bbbk$-module,
because
 \begin{center}
$End_{\Bbbk}(A)=H(A, A)=\oplus_{n\in \mathbb{Z}}H(A, A)_n$
\end{center}
and
 \begin{center}
$H(A, A)_n=\{\psi\in Hom_{\Bbbk}(A, A) | ~\psi(A_i)\subseteq
A_{i+n}\}$.
\end{center}
Let
\begin{center}
$d'\in End(End_{\Bbbk}(A))$, $d'(f)(a)=d(f(a))-(-1)^{|f|}f(d(a))$,
\end{center}
for any elements $f\in End_{\Bbbk}(A)$ and $a\in A$, then
$End_{\Bbbk}(A)$ is a DG algebra. It is easy to proof that $\gamma$
is a graded algebra map. Moreover, we have
\begin{center}
 $(\gamma d)(a)(b)=\gamma(d(a))(b)=d(a)b$
 \end{center}
 and
\begin{center}
 $(d'\gamma)(a)(b)=d'(\gamma(a))(b)=d(\gamma(a)(b))-(-1)^{|\gamma(a)|}\gamma(a)(d(b))=d(ab)-(-1)^{|a|}ad(b),$
\end{center}
for all $a, b\in A$, then $\gamma$ is a DG algebra map by using
graded Leibniz rule for product. In fact, for any homogeneous
elements $a, b, c\in A$, we have
\begin{align*}
[\delta(a), \delta(b)](c)&=\delta(a)\delta(b)(c)-(-1)^{|a||b|}\delta(b)\delta(a)(c)\\
&=\{a, \{b, c\}\}-(-1)^{|a||b|}\{b, \{a, c\}\},\\
\delta(\{a, b\})(c)&=\{\{a, b\}, c\},\\
(\delta d)(a)(b)&=\delta(d(a))(b)=\{d(a), b\}
\end{align*}
and
\begin{center}
 $(d'\delta)(a)(b)=d'(\delta(a))(b)=d(\delta(a)(b))-(-1)^{|a|}\delta(a)(d(b))=d(\{a, b\})-(-1)^{|a|}\{a, d(b)\}.$
 \end{center}
Note that $A$ is a DG Poisson algebra, by the graded Jacobi
identity and graded Leibniz rule for bracket, we have
\begin{center}
 $\{a, \{b, c\}\}-(-1)^{|a||b|}\{b, \{a, c\}\}=\{\{a, b\}, c\}$
\end{center}
and
\begin{center}
 $d(\{a, b\})-(-1)^{|a|}\{a, d(b)\}=\{d(a), b\}$,
\end{center}
which imply that $\delta$ is a DG Lie algebra map. On the other hand, we see
that
\begin{equation}
\begin{split}
\delta(a)\gamma(b)(c)-(-1)^{|a||b|}\gamma(b)\delta(a)(c)&=\delta(a)(bc)-(-1)^{|a||b|}\gamma(b)\{a,
c\} \\&=\{a, bc\}-(-1)^{|a||b|}b\{a, c\}\\&=\gamma(\{a, b\})(c)
\end{split}\nonumber
\end{equation}
and
\begin{equation}
\begin{split}
\gamma(a)\delta(b)(c)+(-1)^{|a||b|}\gamma(b)\delta(a)(c)&=\gamma(a)\{b,
c\}+(-1)^{|a||b|}\gamma(b)\{a, c\}\\&=a\{b, c\}+(-1)^{|a||b|}b\{a,
c\}\\&=\delta(ab)(c),
\end{split}\nonumber
\end{equation}
for all $a, b, c\in A$. Hence there exists a DG algebra map $\phi$
from $A^{ue}$ into $End_{\Bbbk}(A)$ such that $\phi\alpha=\gamma$
and $\phi\beta=\delta$. If $a\in ker\alpha$ then
$0=\phi\alpha(a)=\gamma(a)$, and so $0=\gamma(a)(1)=a$. It completes
the proof.
\end{proof}

Henceforce, we identify the DG algebra homomorphic image of a
finitely generated DG Poisson algebra $A$ under $\alpha$ to $A$ and
denote $\alpha(a)$ by $a$ for all $a\in A$.

\subsection{Construction of $A^e$} For any DG Poisson algebra $A$
given by generators and relations, we give a ``formula" for
computing the universal enveloping algebra $A^e$ of $A$.

\subsubsection{``Anti-differential"} Let $V$ be a graded
$\Bbbk$-vector space with a homogeneous $\Bbbk$-basis
$\{x_{\alpha}: \alpha\in \Lambda\}$ and

$$R=\frac{T(V)}{\langle x_{\alpha}\otimes
x_{\beta}-(-1)^{|x_{\alpha}||x_{\beta}|}x_{\beta}\otimes x_{\alpha}|
~\forall \alpha, ~\beta \in\Lambda \rangle}$$

be a DG Poisson algebra with differential $d$ and Poisson bracket
$\{\cdot, \cdot\}$. Here $T(V)$ is the tensor algebra of $V$ over $\Bbbk$ and
$|x|$ denotes the degree of the homogeneous element $x$ of $R$.

Now for any $\alpha\in \Lambda$, we define a $\Bbbk$-linear map
\begin{center}
$\psi_{\alpha}: R\rightarrow R$
\end{center}
such that
\begin{center}
$\psi_{\alpha}(x_{\beta})=\delta_{\alpha\beta}$ \quad and \quad
$\psi_{\alpha}(ab)=a\psi_{\alpha}(b)+(-1)^{|a||b|}b\psi_{\alpha}(a)$,
\end{center}
for all homogeneous elements $a, b\in R$ and $\beta\in \Lambda$. The
$\Bbbk$-linear map are usually called anti-differential of $R$.

\begin{remark}
We have the following two observations:
\begin{itemize}
  \item It is easy to see that $|\psi_{\alpha}|=-|x_{\alpha}|$ for
  any $\alpha\in \Lambda$, and that the $\Bbbk$-linear map
  $\psi_{\alpha}$ for any $\alpha\in \Lambda$ is well-defined on
  $R$ since
  \begin{equation}
\begin{split}
\psi_{\alpha}(ab-(-1)^{|a||b|}ba)&=\psi_{\alpha}(ab)-(-1)^{|a||b|}\psi_{\alpha}(ba)\\
&=a\psi_{\alpha}(b)+(-1)^{|a||b|}b\psi_{\alpha}(a)-(-1)^{|a||b|}(b\psi_{\alpha}(a)+(-1)^{|a||b|}a\psi_{\alpha}(b))\\
&=0.
\end{split}\nonumber
\end{equation}

  \item If we set $\Lambda_{a}:=\{\alpha \in \Lambda| ~\psi_{\alpha}(a)\neq
  0\}$ for any homogeneous element $a\in R$, then $\Lambda_{a}$ is a
  finite set since $R$ obviously has a PBW-basis.
\end{itemize}
\end{remark}

Now we give some basic properties of these ``anti-differentials".
\begin{lemma}\label{lem2.10}
For any homogeneous element $a \in R$ and for all $\alpha, \beta\in
\Lambda$, we have
\begin{center}
\begin{equation}
\psi_{\alpha}(d(a))-d\psi_{\alpha}(a)-\sum_{\beta\in
\Lambda}(-1)^{|\psi_{\beta}(a)|}\psi_{\beta}(a)\psi_{\alpha}(d(x_{\beta}))=0.
\end{equation}
\end{center}
\end{lemma}

\begin{proof}
Observe that formula (2.1) is true on any $x_{\gamma}$ for all
$\gamma\in \Lambda$ since
\begin{eqnarray*}
LHS ~of (2.1)
&=&\psi_{\alpha}(d(x_{\gamma}))-d\psi_{\alpha}(x_{\gamma})-\sum_{\beta\in
\Lambda}(-1)^{|\psi_{\beta}(x_{\gamma})|}\psi_{\beta}(x_{\gamma})\psi_{\alpha}(d(x_{\beta}))\\
&=&\psi_{\alpha}(d(x_{\gamma}))-\psi_{\alpha}(d(x_{\gamma}))\\
&=&0.
\end{eqnarray*}

In order to prove that formula (2.1) is true for any homogeneous
element $a\in R$. It suffices to prove the formula (2.1) is true for
$ab$ provided that it is true for any homogeneous elements $a, b\in
R$. Thus, assume that we have the following two equations:
\begin{eqnarray*}
\psi_{\alpha}(d(a))-d\psi_{\alpha}(a)-\sum_{\beta\in
\Lambda}(-1)^{|\psi_{\beta}(a)|}\psi_{\beta}(a)\psi_{\alpha}(d(x_{\beta}))=0,\\
\psi_{\alpha}(d(b))-d\psi_{\alpha}(b)-\sum_{\beta\in
\Lambda}(-1)^{|\psi_{\beta}(b)|}\psi_{\beta}(b)\psi_{\alpha}(d(x_{\beta}))=0,\\
\end{eqnarray*}
for any homogeneous elements $a, b\in R$. We have
\begin{equation}
\begin{split}
&\psi_{\alpha}(d(ab))-d\psi_{\alpha}(ab)-\sum_{\beta\in
\Lambda}(-1)^{|\psi_{\beta}(ab)|}\psi_{\beta}(ab)\psi_{\alpha}(d(x_{\beta}))\\
=~~~~&\psi_{\alpha}(d(a)b+(-1)^{|a|}ad(b))-d(a\psi_{\alpha}(b)+(-1)^{|a||b|}b\psi_{\alpha}(a))\\
&-\sum_{\beta\in \Lambda}(-1)^{|\psi_{\beta}(ab)|}(a\psi_{\beta}(b)+(-1)^{|a||b|}b\psi_{\beta}(a))\psi_{\alpha}(d(x_{\beta}))\\
=~~~~&d(a)\psi_{\alpha}(b)+(-1)^{(|a|+1)|b|}b\psi_{\alpha}(d(a))+(-1)^{|a|}a\psi_{\alpha}(d(b))+(-1)^{|a|+|a|(|b|+1)}d(b)\psi_{\alpha}(a)\\
&-d(a)\psi_{\alpha}(b)-(-1)^{|a|}ad\psi_{\alpha}(b)-(-1)^{|a||b|}d(b)\psi_{\alpha}(a)-(-1)^{(|a|+1)|b|}bd\psi_{\alpha}(a)\\
&-\sum_{\beta\in
\Lambda}(-1)^{|\psi_{\beta}(ab)|}a\psi_{\beta}(b)\psi_{\alpha}(d(x_{\beta}))-\sum_{\beta\in
\Lambda}(-1)^{|\psi_{\beta}(ab)|+|a||b|}
b\psi_{\beta}(a)\psi_{\alpha}(d(x_{\beta}))\\
=~~~~&(-1)^{(|a|+1)|b|}b[\psi_{\alpha}(d(a))-d\psi_{\alpha}(a)-\sum_{\beta\in
\Lambda}(-1)^{|\psi_{\beta}(a)|}\psi_{\beta}(a)\psi_{\alpha}(d(x_{\beta}))]\\
&+(-1)^{|a|}a[\psi_{\alpha}(d(b))-d\psi_{\alpha}(b)-\sum_{\beta\in \Lambda}
(-1)^{|\psi_{\beta}(b)|}\psi_{\beta}(b)\psi_{\alpha}(d(x_{\beta}))]\\
=~~~~&0,
\end{split}\nonumber
\end{equation}
as required.
\end{proof}

\begin{lemma}\label{lem2.11}
We have
\begin{equation}
\{a, b\}=\sum_{\alpha\in \Lambda}\psi_{\alpha}(a)\{x_{\alpha}, b\},
\end{equation}
for any homogeneous elements $a, b\in R$.
\end{lemma}

\begin{proof}
It is clear that formula (2.2) is true on the generators
$\{x_{\alpha}\}_{\alpha\in \Lambda}$ of R. Thus in order to complete
the proof, we only need to show
\begin{center}
$\{aa', bb'\}=\sum_{\alpha\in
\Lambda}\psi_{\alpha}(aa')\{x_{\alpha}, bb'\}$
\end{center}
 provided that
\begin{align*}
\{a, b\}&=\sum_{\alpha\in \Lambda}\psi_{\alpha}(a)\{x_{\alpha}, b\}
\end{align*}
and
\begin{align*}
\{a', b'\}&=\sum_{\alpha\in \Lambda}\psi_{\alpha}(a')\{x_{\alpha}
b'\},
\end{align*}
for any homogeneous elements $a, a', b, b'\in R$.

Indeed, we have
\begin{equation}
\begin{split}
\{aa', bb'\}=&\{aa', b\}b'+(-1)^{|b|(|a|+|a'|)}b\{aa', b'\}\\
=&a\{a', b\}b'+(-1)^{|a||a'|}a'\{a, b\}b'+(-1)^{|b|(|a|+|a'|)}ba\{a', b'\}+(-1)^{|b|(|a|+|a'|)+|a||a'|}ba'\{a, b'\}\\
=&\sum_{\alpha\in \Lambda}a\psi_{\alpha}(a')\{x_{\alpha},
b\}b'+(-1)^{|a||a'|}\sum_{\alpha\in \Lambda}a'\psi_{\alpha}(a)\{x_{\alpha}, b\}b'\\
&+(-1)^{|b|(|a|+|a'|)}\sum_{\alpha\in
\Lambda}ba\psi_{\alpha}(a')\{x_{\alpha}, b'\}+(-1)^{|b|(|a|+|a'|)+|a||a'|}\sum_{\alpha\in \Lambda}ba'\psi_{\alpha}(a)\{x_{\alpha}, b'\}\\
=&\sum_{\alpha\in \Lambda}a\psi_{\alpha}(a')\{x_{\alpha},
b\}b'+(-1)^{|b||x_{\alpha}|}\sum_{\alpha\in \Lambda}a\psi_{\alpha}(a')b\{x_{\alpha}, b'\}\\
&+(-1)^{|a||a'|}\sum_{\alpha\in
\Lambda}a'\psi_{\alpha}(a)\{x_{\alpha},
b\}b'+(-1)^{|b||x_{\alpha}|+|a||a'|}\sum_{\alpha\in \Lambda}a'\psi_{\alpha}(a)b\{x_{\alpha}, b'\}\\
=&\sum_{\alpha\in
\Lambda}[(a\psi_{\alpha}(a')+(-1)^{|a||a'|}a'\psi_{\alpha}(a))(\{x_{\alpha},
b\}b'+(-1)^{|b||x_{\alpha}|}b\{x_{\alpha}, b'\})]\\
=&\sum_{\alpha\in \Lambda}\psi_{\alpha}(aa')\{x_{\alpha}, bb'\},
\end{split}\nonumber
\end{equation}
as required.
\end{proof}

\subsubsection{} Now we introduce another set of indeterminates $\{y_{\alpha}|~\alpha\in
 \Lambda\}$ such that $|x_{\alpha}|=|y_{\alpha}|$ for all $\alpha\in
 \Lambda$. Define a graded $R$-free algebra $F(R)$ by
\begin{center}
$F(R):=R\langle y_{\alpha}|~\alpha\in \Lambda \rangle$
\end{center}
and a $\Bbbk$-linear map
\begin{center}
$\psi: R\rightarrow F(R)$
\end{center}
by
\[\psi(f):=\sum_{\alpha\in \Lambda}\psi_{\alpha}(f)y_{\alpha}\]
for any $f\in R$. Note that such $\psi$ is well-defined since for
any $f\in R$, there are only finite many $\alpha\in \Lambda$ such
that $\psi_{\alpha}(f)$ is not zero. Also, let $j: R\rightarrow F(R)$
be the canonical inclusion map.

Now suppose that $(R, d, \{\cdot, \cdot\})$ is a DG Poisson algebra and $I$ is a
DG Poisson ideal of $R$. Put $A:=R/I$, then $A$ has a natural DG
Poisson algebra structure induced from $R$. Let $\pi'$ denote the canonical
projection $\pi': R\rightarrow A$, then $\pi'$ is a DG Poisson
algebra map.

As for $F(R)$, let $J$ be the graded ideal of $F(R)$ generated by
\begin{center}
\begin{enumerate}
  \item $~~~~~~~~~~~~I, ~\psi(I)$,
  \item $~~~~~~~~~~~~y_{\alpha}y_{\beta}-(-1)^{|x_{\alpha}||x_{\beta}|}y_{\beta}y_{\alpha}-\psi(\{x_{\alpha},
x_{\beta}\})$,
  \item $~~~~~~~~~~~~y_{\alpha}x_{\beta}-(-1)^{|x_{\alpha}||x_{\beta}|}x_{\beta}y_{\alpha}-\{x_{\alpha},
x_{\beta}\}$,
\end{enumerate}
\end{center}
where $|x_{\alpha}|=|y_{\alpha}|$ for any $\alpha\in \Lambda$, and
\[\psi(f):=\sum_{\alpha\in \Lambda}\psi_{\alpha}(f)y_{\alpha},\]
for any $f\in R$. This induces a canonical projection $\pi: F(R)\rightarrow \mathcal{A}:=F(R)/J$. For the sake of simplicity, we omit the canonical projections $\pi$ and $\pi'$ here, and also denote by $d$ the differential of $A$.

\begin{lemma}\label{lem1}
$\mathcal{A}$ is a DG algebra.
\end{lemma}

\begin{proof}
Note that $|\psi|=0$, it is easy to see that $\mathcal{A}$ is a graded
algebra from its construction. Define a $\Bbbk$-linear map
$\partial: \mathcal{A}\rightarrow \mathcal{A}$ by
\[\partial x_{\alpha}:=j(d(x_{\alpha}))=d(x_{\alpha}), ~~~\partial y_{\alpha}:=\psi(d(x_{\alpha}))=\sum_{\beta\in \Lambda}\psi_{\beta}d(x_{\alpha})y_{\beta}\]
for all $\alpha\in \Lambda$, and the graded Leibniz rule,
i.e.,
\begin{center}
$\partial(ab)=\partial(a)b+(-1)^{|a|}a\partial(b)$
\end{center}
for all homogeneous elements $a, b\in \mathcal{A}$. It is obvious that
$|\partial|=1$. In order to prove $\mathcal{A}$ is a DG algebra, it
suffices to prove that $\partial^2=0$.

Indeed, it is easy to see that $\partial(f)=d(f)$ for all
$f\in A$ by using induction. We have
\begin{center}
$\partial^2(x_{\alpha})=\partial(d(x_{\alpha}))=d^2(x_{\alpha})=0$,
\end{center}
for all $\alpha\in \Lambda$. Moreover, we have
\begin{equation}
\begin{split}
\partial^2(y_{\alpha})=&\partial(\sum_{\beta\in
\Lambda}\psi_{\beta}d(x_{\alpha})y_{\beta})\\=&\sum_{\beta\in
\Lambda}[\partial\psi_{\beta}d(x_{\alpha})y_{\beta}+(-1)^{|\psi_{\beta}d(x_{\alpha})|}
\psi_{\beta}d(x_{\alpha})\partial(y_{\beta})]\\
=&\sum_{\beta\in \Lambda}[d\psi_{\beta}d(x_{\alpha})y_{\beta}+(-1)^{|\psi_{\beta}d(x_{\alpha})|}
\psi_{\beta}d(x_{\alpha})\sum_{\gamma\in \Lambda}\psi_{\gamma}d(x_{\beta})y_{\gamma}]\\
=&\sum_{\beta\in \Lambda}\psi_{\beta}d^2(x_{\alpha})y_{\beta}-\sum_{\beta, \gamma\in \Lambda}(-1)^{|\psi_{\gamma}d(x_{\alpha})|}\psi_{\gamma}d(x_{\alpha})\psi_{\beta}d(x_{\gamma})y_{\beta}
+\sum_{\beta, \gamma\in \Lambda}(-1)^{|\psi_{\beta}d(x_{\alpha})|}\psi_{\beta}d(x_{\alpha})\psi_{\gamma}d(x_{\beta})y_{\gamma}\\
=&-\sum_{\beta, \gamma\in \Lambda}(-1)^{|\psi_{\gamma}d(x_{\alpha})|}\psi_{\gamma}d(x_{\alpha})\psi_{\beta}d(x_{\gamma})y_{\beta}
+\sum_{\beta, \gamma\in \Lambda}(-1)^{|\psi_{\beta}d(x_{\alpha})|}\psi_{\beta}d(x_{\alpha})\psi_{\gamma}d(x_{\beta})y_{\gamma}\\
=&~~0
\end{split}\nonumber
\end{equation}
by Lemma \ref{lem2.10}.
\end{proof}

Note that there are two $\Bbbk$-linear maps $m, h: A\rightarrow
\mathcal{A}$ given by
\begin{center}
$m(f)=f, ~h(f)=\psi(f)$,
\end{center}
for all $f\in A$. We can obtain our main result in this subsection.

\begin{theorem}
$(\mathcal{A}, m, h)$ is the universal enveloping algebra of a DG Poisson
algebra $A$, where $A=R/I$ is defined as above.
\end{theorem}

\begin{proof}
By Lemma \ref{lem1}, $(\mathcal{A}, \partial)$ is a DG algebra. From the
construction of $\mathcal{A}$, there are two $\Bbbk$-linear maps
$m: A\rightarrow \mathcal{A}$ sending each element $f\in A$ to
$f$ and $h: A\rightarrow \mathcal{A}$ sending each element
$f\in A$ to $\psi(f)$. In fact, $m$ is the DG algebra map
and $h$ is DG $\Bbbk$-linear map since for all homogeneous $f, g\in
A$, we have
\begin{align*}
m(fg)&=fg=m(f)m(g), \\
\partial m(f)&=\partial(f)=d(f)=md(f)
\end{align*}
and
\begin{align*}
\partial h(f)&=\partial(\sum_{\alpha\in
\Lambda}\psi_{\alpha}(f)y_{\alpha})\\
&=\sum_{\alpha\in \Lambda}[\partial\psi_{\alpha}(f)y_{\alpha}+(-1)^{|\psi_{\alpha}(f)|}\psi_{\alpha}(f)\partial(y_{\alpha})]\\
&=\sum_{\alpha\in \Lambda}[d\psi_{\alpha}(f)y_{\alpha}+(-1)^{|\psi_{\alpha}(f)|}\psi_{\alpha}(f)({\sum_{\beta\in \Lambda}\psi_{\beta}d(x_{\alpha})y_{\beta}})]\\
&=\sum_{\alpha\in \Lambda}[d\psi_{\alpha}(f)y_{\alpha}]+\sum_{\beta\in \Lambda}[\psi_{\beta}d(f)-d\psi_{\beta}(f)]y_{\beta}\\
&=hd(f)
\end{align*}
by Lemma \ref{lem2.10}.

Notice that for any homogeneous elements $f, g\in A$, we have
\begin{equation}
\begin{split}
h(fg)&=\psi(fg)=\sum_{\alpha\in
\Lambda}\psi_{\alpha}(fg)y_{\alpha}=\sum_{\alpha\in
\Lambda}f\psi_{\alpha}(g)y_{\alpha}+\sum_{\alpha\in
\Lambda}(-1)^{|f||g|}g\psi_{\alpha}(f)y_{\alpha}\\
&=f\psi(g)+(-1)^{|f||g|}g\psi(f)=m(f)h(g)+(-1)^{|f||g|}m(g)h(f).
\end{split}\nonumber
\end{equation}
For a monomial $f\in A$, we proceed the proof using induction on the
length of $f$. We already know that
\begin{equation}
\begin{split}
h(x_{\alpha})m(x_{\beta})-(-1)^{|x_{\alpha}||x_{\beta}|}m(x_{\beta})h(x_{\alpha})&=
\sum_{\gamma\in
\Lambda}\psi_{\gamma}(x_{\alpha})y_{\gamma}x_{\beta}-(-1)^{|x_{\alpha}||x_{\beta}|}x_{\beta}
\sum_{\gamma\in
\Lambda}\psi_{\gamma}(x_{\alpha})y_{\gamma}\\
&=y_{\alpha}x_{\beta}-(-1)^{|x_{\alpha}||x_{\beta}|}x_{\beta}y_{\alpha}\\
&=m(\{x_{\alpha},x_{\beta}\}).
\end{split}\nonumber
\end{equation}
For any monomials $f, g, a\in A$, assume that
\begin{center}
$m(\{f, a\})=h(f)m(a)-(-1)^{|f||a|}m(a)h(f)$
\end{center}
and
\begin{center}
$m(\{g, a\})=h(g)m(a)-(-1)^{|g||a|}m(a)h(g)$.
\end{center}
Now we have
\begin{align*}
&m(\{fg, a\})\\=&m(f\{g, a\}+(-1)^{|f||g|}g\{f, a\})\\
=&m(f)[h(g)m(a)-(-1)^{|g||a|}m(a)h(g)]+(-1)^{|f||g|}m(g)[h(f)m(a)-(-1)^{|f||a|}m(a)h(f)]\\
=&[m(f)h(g)+(-1)^{|f||g|}m(g)h(f)]m(a)-(-1)^{|f||a|+|g||a|}m(a)[m(f)h(g)+(-1)^{|f||g|}m(g)h(f)]\\
=&h(fg)m(a)-(-1)^{|fg||a|}m(a)h(fg).
\end{align*}
Applying the induction, we have
\begin{center}
$m(\{f, g\})=h(f)m(g)-(-1)^{|f||g|}m(g)h(f)$,
\end{center}
for all $f, g\in A$. Similarly, we have $h(\{x_{\alpha},
x_{\beta}\})=[h(x_{\alpha}), h(x_{\beta})]$. For any
monomials $f, g, a\in A$, assume that
\begin{center}
$h(\{a, f\})=[h(a), h(f)]$, \quad $h(\{a, g\})=[h(a), h(g)]$.
\end{center}
Then we have
\begin{align*}
&h(\{a, fg\})\\=&h(\{a, f\}g+(-1)^{|a||f|}f\{a, g\})\\
=&m(\{a, f\})h(g)+(-1)^{|\{a, f\}||g|}m(g)h(\{a,
f\})+(-1)^{|a||f|}[m(f)h(\{a, g\})+(-1)^{|f||\{a, g\}|}m(\{a,
g\})h(f)]\\
=&[h(a)m(f)-(-1)^{|a||f|}m(f)h(a)]h(g)+(-1)^{(|a|+|f|)|g|}m(g)[h(a)h(f)-(-1)^{|a||f|}h(f)h(a)]\\
&+(-1)^{|a||f|}m(f)[h(a)h(g)-(-1)^{|a||g|}h(g)h(a)]+(-1)^{|f||g|}[h(a)m(g)-(-1)^{|a||g|}m(g)h(a)]h(f)\\
=&h(a)m(f)h(g)-(-1)^{|a||f|+|f||g|+|g||a|}m(g)h(f)h(a)\\
&-(-1)^{|a||f|+|a||g|}m(f)h(g)h(a)+(-1)^{|f||g|}h(a)m(g)h(f)\\
=&h(a)h(fg)-(-1)^{|a||f|+|a||g|}h(fg)h(a)\\
=&[h(a), h(fg)].
\end{align*}
Applying the induction, we have
\begin{center}
$h(\{f, g\})=[h(f), h(g)]$,
\end{center}
for all $f, g\in A$.

Let $(D, \partial')$ be a DG algebra. Let $\gamma: (A, d)\rightarrow
(D, \partial')$ be a DG algebra map and let $\delta: (A, \{\cdot, \cdot\},
d)\rightarrow (D_P, [\cdot, \cdot], \partial')$ be a DG Lie algebra map such
that
\begin{align*}
\gamma(\{a, b\})&=\delta(a)\gamma(b)-(-1)^{|a||b|}\gamma(b)\delta(a),\\
\delta(ab)&=\gamma(a)\delta(b)+(-1)^{|a||b|}\gamma(b)\delta(a),
\end{align*}
for any homogeneous elements $a, b\in A$. Since $F(R)$ is a graded
free algebra, there exists a graded algebra map $\phi': F(R) \rightarrow D$
defined by
\begin{center}
$\phi'(x_{\alpha})=\gamma(x_{\alpha})$,
$~~\phi'(y_{\alpha})=\delta(x_{\alpha})$,
\end{center}
for all $\alpha\in \Lambda$.

It is shown to be $\phi'\psi=\delta$ using induction on the
length of monomials in $R$. Therefore we have
\begin{align*}
&\phi'(I)=\gamma(I)=0, \quad \quad \phi'(\psi(I))=\delta(I)=0,\\
&\phi'(y_{\alpha}y_{\beta}-(-1)^{|x_{\alpha}||x_{\beta}|}y_{\beta}y_{\alpha}-\psi(\{x_{\alpha},
x_{\alpha}\}))\\
=&\delta(x_{\alpha})\delta(x_{\beta})-(-1)^{|x_{\alpha}||x_{\beta}|}\delta(x_{\beta})\delta(x_{\alpha})-\delta(\{x_{\alpha},x_{\beta}\})\\
=&[\delta(x_{\alpha}),\delta(x_{\beta})]-\delta(\{x_{\alpha}, x_{\beta}\})=0,\\
&\phi'(y_{\alpha}x_{\beta}-(-1)^{|x_{\alpha}||x_{\beta}|}x_{\beta}y_{\alpha}-\{x_{\alpha},
x_{\beta}\})\\
=&\delta(x_{\alpha})\gamma(x_{\beta})-(-1)^{|x_{\alpha}||x_{\beta}|}\gamma(x_{\beta})\delta(x_{\alpha})-\gamma(\{x_{\alpha},x_{\beta}\})\\
=&\gamma(\{x_{\alpha}, x_{\beta}\})-\gamma(\{x_{\alpha},
x_{\beta}\})=0.
\end{align*}
Thus there exists a graded algebra map $\phi$ from $\mathcal{A}$ into $D$
such that $\phi m=\gamma, \phi h=\delta$.

Further, A graded algebra map $\phi: \mathcal{A}\rightarrow D$ is unique
such that $\phi m=\gamma$, $\phi h=\delta$ by its construction and
it is also a DG algebra map since for any $\alpha\in \Lambda$, we
have
\begin{center}
$\phi\partial(x_{\alpha})=\phi d(x_{\alpha})=\phi
m(d(x_{\alpha}))=\gamma d(x_{\alpha})
=\partial'\gamma(x_{\alpha})=\partial'\phi(x_{\alpha})$
\end{center}
and
\begin{center}
$\phi\partial(y_{\alpha})=\phi\psi(d(x_{\alpha}))=\phi
h(d(x_{\alpha}))=\delta d(x_{\alpha})
=\partial'\delta(x_{\alpha})=\partial'\phi(y_{\alpha})$.
\end{center}
Therefore, $(\mathcal{A}, m, h)$ is the universal enveloping algebra of
a DG Poisson algebra $A$, as required.
\end{proof}

\begin{remark}
From now on, we let $A^e:=\mathcal{A}$ denote the universal enveloping
algebra of a DG Poisson algebra $A$ given by generators and
relations.
\end{remark}

\medskip
\section{Poincar\'{e}-Birkhoff-Witt theorem for universal
enveloping algebras} In this section, we develop an algorithm to
find a $\Bbbk$-linear basis of the universal enveloping algebra
$A^e$.

\subsection{Gr\"{o}bner-Shirshov basis theory}
In this subsection, we recall the Gr\"{o}bner-Shirshov basis theory
developed in \cite{KL} and \cite{KLe}.

 Let $X=\{x_1, x_2, \cdots\}$ be a set of alphabets indexed by positive
integers. Define a linear ordering on $X$ $\prec$ by setting
$x_i\prec x_j$ if and only if $i<j$. Let $X^*$ be the free monoid of
associative monomials on $X$. We denote the empty monomial by 1 and
the length of a monomial $u$ by $l(u)$ with $l(1)=0$. We consider
two linear ordering $<$ and $\ll$ on $X^*$ defined as follows
\cite{KLe}:
\begin{enumerate}
  \item[(\rmnum{1})] $1<u$ for any nonempty monomial $u$; and
  inductively, $u<v$ whenever $u=x_iu', v=x_jv'$ and $x_i\prec x_j$ or
  $x_i=x_j$ and $u'<v'$ with $x_i, x_j\in X$,
  \item[(\rmnum{2})] $u\ll v$ if $l(u)<l(v)$ or $l(u)=l(v)$ and $u<v$.
\end{enumerate}
The ordering $<$ (resp. $\ll$) is called the lexicographic ordering
(resp. degree-lexicographic ordering). It is easy to see $\ll$ is a
monomial order on $X^*$, that is, $x\ll y$ implies $axb\ll ayb$ for
all $a, b\in X^*$.

Let $T_X$ be the DG free $k$-algebra generated by $X$, let $I$ be a
DG ideal of $T_X$ and let $T_0=T_X/I$, then $T_0$ is a DG algebra.
The image of $p\in T_X$ in $T_0$ under the canonical DG quotient map
will also be denoted by $p$. Given a nonzero element $p\in T_0$, we
denote by $\bar{p}$ the maximal monomial appearing in $p$ under the
ordering $\ll$. Thus $p=\alpha\bar{p}+\Sigma \beta_iw_i$ with
$\alpha, \beta_i\in \Bbbk$, $w_i\in X^*$, $\alpha\neq 0$ and $w_i\ll
\bar{p}$. The $\alpha$ is called the leading coefficient of $p$ and if
$\alpha=1$, then $p$ is said to be monic. Recall that the
composition of $p$ and $q$ as follows.
\begin{defn} \cite{KL}
Let $p$ and $q$ be monic elements of $T_0$.
\begin{enumerate}
  \item[(a)]If there exist $a$ and $b$ in $X^*$ such that
$\bar{p}a=b\bar{q}=w$ with $l(\bar{p})>l(b)$, then the composition
of intersection is defined to be $(p, q)_w=pa-bq$.
  \item[(b)]If there exist $a$ and $b$ in $X^*$ such that $a\neq 1$,
$a\bar{p}b=\bar{q}=w$, then the composition of inclusion is defined
to be $(p, q)_w=apb-q$.
\end{enumerate}
\end{defn}
Let $S$ be a subset of monic elements of $T_0$ and let $J$ be the DG
ideal of $T_0$ generated by $S$. Then we say the DG algebra $T_0/J$
is defined by $S$. The image of $p\in T_0$ in $T_0/J$ under the
canonical quotient maps will also be denoted by $p$ as long as there
is no peril of confusion. Next, let $p, q\in T_0$ and $w\in X^*$. We
define a congruence relation on $T_0$ as follows: $p\equiv q$ mod
$(J; w)$ if and only if $p-q=\Sigma\alpha_ia_is_ib_i$, where
$\alpha_i\in \Bbbk$, $a_i, b_i\in X^*$, $s_i\in S$ and $a_i\bar{s}_ib_i\ll w$. A set $S$ of monic elements of $T_0$ is said
to be closed under the composition if for any $p, q\in S$ and $w\in
X^*$ such that $(p, q)_w$ is defined, we have $(p, q)_w\equiv 0$ mod
$(J; w)$.

Now we introduce another definition of monomials and then prove the
following generalization of Shirshov's Composition Lemma to the
representations of DG associative algebras.
\begin{defn} \cite{KL}
 A monomial $u\in X^*$ is said to be $S$-standard in $T_0$ if $u\neq
a\bar{s}b$ for any $s\in S$ and $a, b\in X^*$. Otherwise, the
monomial $u$ is said to be $S$-reducible in $T_0$.
\end{defn}

\begin{theorem}
Let $S$ be a subset of monic elements in $T_0$ and let $T_0/J$ be
the DG algebra is defined by $S$. If $S$ is closed under composition
in $T_0$ and the image of $p\in T_0$ is zero in $T_0/J$, then the
monomial $\bar{p}$ is $S$-reducible in $T_0$.
\end{theorem}

\begin{proof}
Since the image of $p\in T_0$ is zero in $T_0/J$, we have
$p=\Sigma\alpha_ia_is_ib_i$, where $\alpha_i\in \Bbbk$, $a_i, b_i\in
X^*$ and $s_i\in S$. Choose the maximal monomial $w$ in the
degree-lexicographic ordering $\ll$ among the monomials
$\{a_i\overline{s_i}b_i\}$ in the expression of $p$. If $\bar{p}=w$,
then we are done. Suppose this is not the case, then $\bar{p}\ll w$
and without loss of generality, we may assume that the following
case hold: $w=a_1\overline{s_1}b_1=a_2\overline{s_2}b_2$.

If $w=a_1\overline{s_1}b_1=a_2\overline{s_2}b_2$, then we should
show that $a_1s_1b_1\equiv a_2s_2b_2$ mod $(J; w)$. There are three
possibilities:

(i) If the monomials $\overline{s_1}$ and $\overline{s_2}$ have
empty intersection in $w$, then we may assume that
$a_1s_1b_1=as_1b\overline{s_2}c$ and
$a_2s_2b_2=a\overline{s_1}bs_2c$, where $a, b, c\in X^*$. Thus
\begin{center}
$a_2s_2b_2-a_1s_1b_1=-a(s_1-\overline{s_1})bs_2c+as_1b(s_2-\overline{s_2})c$,
\end{center}
which implies $a_1s_1b_1\equiv a_2s_2b_2$ mod $(J; w)$.

(ii) If $\overline{s_1}=u_1u_2$ and $\overline{s_2}=u_2u_3$ for some
$u_2\neq 1$, then $a_2=a_1u_1$, $b_1=u_3b_2$ and
\begin{align*}
a_2s_2b_2-a_1s_1b_1&=a_1u_1s_2b_2-a_1s_1u_3b_2\\&=-a_1(s_1u_3-u_1s_2)b_2=a_1(s_1,
s_2)_{u_1u_2u_3}b_2.
\end{align*}
Since $\overline{(s_1, s_2)_{u_1u_2u_3}}\ll u_1u_2u_3$ and $S$ is
closed under composition in $T_0$, we obtain $a_1s_1b_1\equiv
a_2s_2b_2$ mod $(J; w)$.

(iii) If $\overline{s_1}=u_1\overline{s_2}u_2$, then $a_2=a_1u_1$,
$b_2=u_2b_1$ and
\begin{align*}
a_2s_2b_2-a_1s_1b_1&=a_1u_1s_2u_2b_1-a_1s_1b_1\\&=a_1(u_1s_2u_2-s_1)b_2=a_1(s_2,
s_1)_{\overline{s_1}}b_1.
\end{align*}
Since $\overline{(s_2, s_1)_{\overline{s_1}}}\ll \overline{s_1}$ and
$S$ is closed under composition in $T_0$, we get $a_1s_1b_1\equiv
a_2s_2b_2$ mod $(J; w)$.

Therefore, $p$ can be written as $p=\Sigma\alpha'_ia'_is'_ib'_i$,
where $a'_i\overline{s'_i}b'_i\ll w$ for all $i$. Choose the maximal
monomial $w_1$ in the ordering $\ll$ among
$\{a'_i\overline{s'_i}b'_i\}$. If $\bar{p}=w_1$, then we are done.
If this is not the case, repeat the above process. Since $X$ is
indexed by the set of positive integers, this process must terminate
in finite steps, which completes the proof.
\end{proof}

As a corollary, we obtain:

\begin{prop}\label{th1}
Let $\mathcal{A}\subseteq X^*$ form a $\Bbbk$-linear basis of DG
algebra $T_0=T_X/I$, let $S$ be a subset of monic elements of $T_0$
and let $T_0/J$ be the DG algebra is defined by $S$. Then the
following are equivalent:
\begin{enumerate}
  \item[(\rmnum{1})] $S$ is closed under composition in $T_0$.
  \item[(\rmnum{2})] the subset of $\mathcal{A}$ consisting of $S$-standard
monomials in $T_0$ forms a $\Bbbk$-linear basis of the DG algebra
$T_0/J$.
\end{enumerate}
\end{prop}

\begin{proof}
Copy the proof of proposition 1.9 in \cite{KL}.
\end{proof}

\subsection{PBW-basis for the universal enveloping algebra}
Let $A$ be a DG Poisson homomorphic image of a DG Poisson algebra
$R$ with an arbitrary differential $d$ and Poisson bracket $\{\cdot, \cdot\}$,
where

$$R=\frac{T(V)}{\langle x_{\alpha}\otimes
x_{\beta}-(-1)^{|x_{\alpha}||x_{\beta}|}x_{\beta}\otimes x_{\alpha}|
~\forall \alpha, ~\beta \in\Lambda \rangle}$$
and $\Lambda$ is a finite index set. For the convenience of the
narrative, assume that $\Lambda=\{1, 2, \cdots, n\}$. In this subsection, we
will find a $\Bbbk$-linear basis of $R^{e}$ and develop an algorithm to
find a $\Bbbk$-linear basis of the universal enveloping algebra
$A^{e}$. From now on, assume that,
\newline

 $\bullet ~~(R, \{\cdot, \cdot\}, d)$ is a DG Poisson algebra, where $R$ is
 defined as above and $\Lambda$ is a finite index set. Assume that
 $\Lambda=\{1, 2, \cdots, n\}$.

 $\bullet ~~A=R/I$, where $I$ is a DG Poisson ideal of $R$.

 $\bullet ~~F(R)=R\langle y_i | ~i\in \Lambda \rangle$.

 $\bullet ~~\psi: R\rightarrow F(R), ~~\psi(f)=\sum_{r=1}^n\psi_r(f)y_r$.

 $\bullet ~~J$ is the graded ideal of $F(R)$ generated
\begin{center}
 $I$, $~\psi(I)$, $~y_ix_j-(-1)^{|x_i||x_j|}x_jy_i-\{x_i, x_j\}$, $~y_iy_j-(-1)^{|x_i||x_j|}y_jy_i-\psi(\{x_i, x_j\})$
\end{center}
for all $i, j=1, 2, \cdots, n$.

 $\bullet ~~A^e=F(R)/J$ is the universal enveloping algebra of $A$ and let $\pi: F(R)\rightarrow A^e$ denote the
 canonical projection.

Now define a linear ordering on the index set $\Lambda\times
\Lambda$ by
 \begin{center}
 $(l, m)<(r, s)\Longleftrightarrow m<s ~~or ~~m=s ~~and ~~l<r$
\end{center}
and a grading on $F(R)$ by
 \begin{center}
$deg(x_i)=(1, 0)$, ~~$deg(y_i)=(0, 1)$
\end{center}
for all $i\in \Lambda$, hence the grading of a monomial
$u=u_1\cdots u_l\in F(R)$, where $u_j=x_i$ or $u_j=y_i$, is defined by
 \begin{center}
 $deg(u)=deg(u_1)+\cdots+deg(u_l)\in \Lambda\times \Lambda$.
 \end{center}
We give an ordering $<$ on the set of generators of $F(R)$ by
 \begin{center}
 $x_1<x_2<\cdots<x_n<y_1<y_2<\cdots<y_n$.
 \end{center}
Therefore there is a well-ordering $\prec$ on the set of all
monomials in $F(R)$. That is, for monomials $u=u_1\cdots u_l$ and
$v=v_1\cdots v_m$, we denote $u\prec v$ if one of the following
conditions holds:
\begin{enumerate}
  \item[(\rmnum{1})] $deg(u)<deg(v)$.
  \item[(\rmnum{2})] $deg(u)=deg(v), u_1=v_1, \cdots, u_r=v_r ~and ~u_{r+1}<v_{r+1}$ for some $r\in
  \Lambda$.
\end{enumerate}
Note that the ordering $\prec$ is a monomial order.

\begin{exa}
For $u=x_2x_3y_3^2y_1$, $v=x_3y_2x_1^2$ and $w=x_2y_3x_3y_2y_1$, we
have $deg(u)=deg(w)=(2, 3)$ and $deg(v)=(3, 1)$, hence $v\prec
u\prec w$.
\end{exa}

\begin{lemma}\label{lem2}
For a monomial $f=x_{i_1}\cdots x_{i_r}\in F(R)$, $y_if\equiv
(-1)^{|x_i||f|}fy_i+\{x_i, f\} ~mod (J; y_if)$.
\end{lemma}

\begin{proof}
We proceed the proof using induction on $l(f)=r$. If $r=$0 or 1, then
it is trivial. Assume that the statement is true for monomials with
length less than $r$. Then we have that
\begin{equation}
\begin{split}
y_if&=y_i(x_{i_1}\cdots x_{i_{r-1}})x_{i_r}\\
&\equiv((-1)^{|x_i||x_{i_1}\cdots x_{i_{r-1}}|}x_{i_1}\cdots x_{i_{r-1}}y_i+\{x_i, x_{i_1}\cdots x_{i_{r-1}}\})x_{i_r}\\
&\equiv(-1)^{|x_i||x_{i_1}\cdots x_{i_{r-1}}|}x_{i_1}\cdots x_{i_{r-1}}((-1)^{|x_i||x_{i_r}|}x_{i_r}y_i+\{x_i,
x_{i_r}\})+\{x_i, x_{i_1}\cdots x_{i_{r-1}}\}x_{i_r}\\
&\equiv(-1)^{|x_i||f|}fy_i+\{x_i, f\} ~mod (J; y_if)
\end{split}\nonumber
\end{equation}
by the induction hypothesis and biderivation property.
\end{proof}

\begin{lemma}\label{lem3}
For a monomial $f=x_{i_1}\cdots x_{i_r}\in F(R)$,
\begin{center}
$\psi(f)x_i\equiv (-1)^{|x_i||f|}x_i\psi(f)+\{f, x_i\}~mod (J;
\overline{\psi(f)}x_i)$.
 \end{center}
\end{lemma}

\begin{proof}
By Lemma \ref{lem2.11} and the definition of $\psi$, we have $\{f, x_i\}=\sum_{s}\psi_s(f)\{x_s, x_i\}$ and $\psi(f)=\sum_{s}\psi_s(f)y_s$. Thus
\begin{equation}
\begin{split}
\psi(f)x_i&=\sum_{s}\psi_{s}(f)y_{s}x_i\\&
\equiv\sum_{s}\psi_{s}(f)((-1)^{|x_{s}||x_i|}x_iy_{s}+\{x_{s},
x_i\})\\
&\equiv\sum_{s}(-1)^{|x_{s}||x_i|}(-1)^{(|f|-|x_{s}|)|x_i|}x_i\psi_{s}(f)y_{s}+\sum_{s}\psi_{s}(f)\{x_{s},
x_i\}\\
&\equiv(-1)^{|f||x_i|}x_i\psi(f)+\sum_{s}\psi_{s}(f)\{x_{s},
x_i\}\\
&\equiv(-1)^{|f||x_i|}x_i\psi(f)+\{f, x_i\} ~mod (J;
\overline{\psi(f)}x_i)
\end{split}\nonumber
\end{equation}
by Lemma \ref{lem2}.
\end{proof}

\begin{lemma}\label{lem4}
For a monomial $f=x_{i_1}\cdots x_{i_r}\in F(R)$,
\begin{align*}
y_i\psi(f)\equiv& (-1)^{|x_i||f|}\psi(f)y_i+\sum_{s,
t}(-1)^{|x_i||\psi_{s}(f)|}\psi_{s}(f)\psi_{t}(f_{is})y_{t}\\
&+\sum_{s,
t}(-1)^{|x_i||\psi_{t}\psi_{s}(f)|}\psi_{t}\psi_{s}(f)f_{it}y_{s} ~mod (J; y_i\overline{\psi(f)}),
\end{align*}
where $f_{ij}=\{x_i, x_j\}$.
\end{lemma}

\begin{proof}
Note that $\psi(f)=\sum_{s}\psi_{s}(f)y_{s}$ and $\{x_i,
f\}=\sum_{s}(-1)^{|x_i||\psi_{s}(f)|}\psi_{s}(f)f_{is}$, we have
\begin{align*}
y_i\psi(f)=&y_i\sum_{s}\psi_{s}(f)y_s\\
\equiv&\sum_{s}((-1)^{|\psi_{s}(f)||x_i|}\psi_{s}(f)y_i+\{x_i,
\psi_{s}(f)\})y_{s}\\
\equiv&\sum_{s}(-1)^{|\psi_{s}(f)||x_i|}\psi_{s}(f)((-1)^{|x_i||x_{s}|}y_{s}y_i+\psi(f_{is}))+\sum_{s}\{x_i, \psi_{s}(f)\}y_{s}\\
\equiv&(-1)^{|x_i||f|}\psi(f)y_i+\sum_{s,
t}(-1)^{|x_i||\psi_{s}(f)|}\psi_{s}(f)\psi_{t}(f_{is})y_{t}\\
&+\sum_{s,
t}(-1)^{|x_i||\psi_{t}\psi_{s}(f)|}\psi_{t}\psi_{s}(f)f_{it}y_{s} ~mod (J; y_i\overline{\psi(f)})
\end{align*}
by Lemma \ref{lem2}.
\end{proof}

\begin{lemma}\label{lem5}
Retain the above notions, we have that
\begin{equation}
\begin{split}
&\sum_{s,
t}(-1)^{|x_i||\psi_{s}(f_{jk})|}[\psi_{s}(f_{jk})\psi_{t}(f_{is})+(-1)^{|\psi_{s}(f_{jk})||f_{is}|}f_{is}\psi_{t}\psi_{s}(f_{jk})]y_t\\
&+\sum_{s,
t}(-1)^{|x_k||f_{ij}|+|x_k||\psi_{s}(f_{ij})|}[\psi_{s}(f_{ij})\psi_{t}(f_{ks})+(-1)^{|\psi_{s}(f_{ij})||f_{ks}|}f_{ks}\psi_{t}\psi_{s}(f_{ij})]y_t\\
&+\sum_{s,
t}(-1)^{|x_i||f_{jk}|+|x_j||\psi_{s}(f_{ki})|}[\psi_{s}(f_{ki})\psi_{t}(f_{js})+(-1)^{|\psi_{s}(f_{ki})||f_{js}|}f_{js}\psi_{t}\psi_{s}(f_{ki})]y_t
~\equiv 0 ~mod J,
\end{split}\nonumber
\end{equation}
where $f_{ij}=\{x_i, x_j\}$.
\end{lemma}

\begin{proof}
Since $\{f, x_i\}=\sum_{s}\psi_{s}(f)\{x_s, x_i\}$, we have
\begin{equation}
\begin{split}
0\equiv&\{x_i, f_{jk}\}+(-1)^{|x_k||f_{ij}|}\{x_k, f_{ij}\}+(-1)^{|x_i||f_{jk}|}\{x_j, f_{ki}\}\\
\equiv&\sum_{s}(-1)^{|x_i||\psi_{s}(f_{jk})|}\psi_{s}(f_{jk})f_{is}+(-1)^{|x_k||f_{ij}|}
\sum_{s}(-1)^{|x_k||\psi_{s}(f_{ij})|}\psi_{s}(f_{ij})f_{ks}\\
&+(-1)^{|x_i||f_{jk}|}\sum_{s}(-1)^{|x_j||\psi_{s}(f_{ki})|}\psi_{s}(f_{ki})f_{js}.
\end{split}\nonumber
\end{equation}
Hence
\begin{equation}
\begin{split}
0\equiv&\psi[\sum_{s}(-1)^{|x_i||\psi_{s}(f_{jk})|}\psi_{s}(f_{jk})f_{is}+(-1)^{|x_k||f_{ij}|}
\sum_{s}(-1)^{|x_k||\psi_{s}(f_{ij})|}\psi_{s}(f_{ij})f_{ks}]\\
&+\psi[(-1)^{|x_i||f_{jk}|}\sum_{s}(-1)^{|x_j||\psi_{s}(f_{ki})|}\psi_{s}(f_{ki})f_{js}]\\
=&\sum_{s}(-1)^{|x_i||\psi_{s}(f_{jk})|}\sum_{t}\psi_{t}(\psi_{s}(f_{jk})f_{is})y_t+(-1)^{|x_k||f_{ij}|}
\sum_{s}(-1)^{|x_k||\psi_{s}(f_{ij})|}\sum_{t}\psi_{t}(\psi_{s}(f_{ij})f_{ks})y_t\\
&+(-1)^{|x_i||f_{jk}|}\sum_{s}(-1)^{|x_j||\psi_{s}(f_{ki})|}\sum_{t}\psi_{t}(\psi_{s}(f_{ki})f_{js})y_t\\
=&\sum_{s,
t}(-1)^{|x_i||\psi_{s}(f_{jk})|}[\psi_{s}(f_{jk})\psi_{t}(f_{is})+(-1)^{|\psi_{s}(f_{jk})||f_{is}|}f_{is}\psi_{t}\psi_{s}(f_{jk})]y_t\\
&+\sum_{s,
t}(-1)^{|x_k||f_{ij}|+|x_k||\psi_{s}(f_{ij})|}[\psi_{s}(f_{ij})\psi_{t}(f_{ks})+(-1)^{|\psi_{s}(f_{ij})||f_{ks}|}f_{ks}\psi_{t}\psi_{s}(f_{ij})]y_t\\
&+\sum_{s,
t}(-1)^{|x_i||f_{jk}|+|x_j||\psi_{s}(f_{ki})|}[\psi_{s}(f_{ki})\psi_{t}(f_{js})+(-1)^{|\psi_{s}(f_{ki})||f_{js}|}f_{js}\psi_{t}\psi_{s}(f_{ki})]y_t ~mod J.\\
\end{split}\nonumber
\end{equation}
\end{proof}

\begin{lemma}\label{lem+}
For a monomial $f=x_{i_1}\cdots x_{i_r}\in F(R)$,
\[\sum_{s,
t}(-1)^{|x_s||x_t|}\psi_{s}\psi_{t}(f)f_{it}y_s=\sum_{s,
t}\psi_{t}\psi_{s}(f)f_{it}y_s,\]
where $f_{ij}=\{x_i, x_j\}$.
\end{lemma}

\begin{proof}
We proceed the proof using induction on $l(f)=r$. If $r=1$ or $2$, then
it is trivial. Assume that the statement is true for monomials with
length less than $r$. Set $f_{ij}=\{x_i, x_j\}$, then we have that
\begin{align*}
&\sum_{s,
t}(-1)^{|x_s||x_t|}\psi_{s}\psi_{t}(x_{i_1}\cdots x_{i_{r-1}}\cdot x_{i_r})f_{it}y_s\\
=&\sum_{s,
t}(-1)^{|x_s||x_t|}\psi_{s}[x_{i_1}\cdots x_{i_{r-1}}\psi_{t}(x_{i_r})+(-1)^{|x_{i_r}||x_{i_1}\cdots x_{i_{r-1}}|}x_{i_r}\psi_{t}(x_{i_1}\cdots x_{i_{r-1}})]f_{it}y_s\\
=&\sum_{s, t}(-1)^{|x_s||x_t|}[x_{i_1}\cdots x_{i_{r-1}}\psi_{s}\psi_{t}(x_{i_r})+(-1)^{|x_{i_1}\cdots x_{i_{r-1}}||\psi_t(x_{i_r})|}\psi_{t}(x_{i_r})
\psi_{s}(x_{i_1}\cdots x_{i_{r-1}})\\&+(-1)^{|x_{i_r}||x_{i_1}\cdots x_{i_{r-1}}|}(x_{i_r}\psi_{s}\psi_{t}(x_{i_1}\cdots x_{i_{r-1}})+(-1)^{|x_{i_r}||\psi_{t}(x_{i_1}\cdots x_{i_{r-1}})|}
\psi_{t}(x_{i_1}\cdots x_{i_{r-1}})\psi_{s}(x_{i_r}))]f_{it}y_s
\end{align*}
and
\begin{align*}
&\sum_{s,
t}\psi_{t}\psi_{s}(x_{i_1}\cdots x_{i_{r-1}}\cdot x_{i_r})f_{it}y_s\\
=&\sum_{s,
t}\psi_{t}[x_{i_1}\cdots x_{i_{r-1}}\psi_{s}(x_{i_r})+(-1)^{|x_{i_r}||x_{i_1}\cdots x_{i_{r-1}}|}x_{i_r}\psi_{s}(x_{i_1}\cdots x_{i_{r-1}})]f_{it}y_s\\
=&\sum_{s,
t}[x_{i_1}\cdots x_{i_{r-1}}\psi_{t}\psi_{s}(x_{i_r})+(-1)^{|x_{i_1}\cdots x_{i_{r-1}}||\psi_s(x_{i_r})|}\psi_{s}(x_{i_r})
\psi_{t}(x_{i_1}\cdots x_{i_{r-1}})]f_{it}y_s\\&+\sum_{s,
t}(-1)^{|x_{i_r}||x_{i_1}\cdots x_{i_{r-1}}|}[x_{i_r}\psi_{t}\psi_{s}(x_{i_1}\cdots x_{i_{r-1}})+(-1)^{|x_{i_r}||\psi_{s}(x_{i_1}\cdots x_{i_{r-1}})|}
\psi_{s}(x_{i_1}\cdots x_{i_{r-1}})\psi_{t}(x_{i_r})]f_{it}y_s.
\end{align*}
Thus
\[\sum_{s,
t}(-1)^{|x_s||x_t|}\psi_{s}\psi_{t}(x_{i_1}\cdots x_{i_{r-1}}\cdot x_{i_r})f_{it}y_s=\sum_{s,
t}\psi_{t}\psi_{s}(x_{i_1}\cdots x_{i_{r-1}}\cdot x_{i_r})f_{it}y_s\]
by the induction hypothesis. Therefore, we finish the proof.
\end{proof}

\begin{lemma}\label{lem7}
In $(A^e, m, h)$, we have
\begin{enumerate}
  \item[(a)] $y_if=(-1)^{|x_i||f|}fy_i+\{x_i, f\}$ for $f\in A$,
  \item[(b)] $h(f)x_i=(-1)^{|x_i||f|}x_ih(f)+\{f, x_i\}$ for $f\in A$,
  \item[(c)] $y_ih(f)=(-1)^{|x_i||f|}h(f)y_i+h(\{x_i, f\})$ for $f\in A$.
\end{enumerate}
\end{lemma}

\begin{proof}
From Lemmas \ref{lem2} and \ref{lem3}, it is easy to see (a) and (b).
Since $\psi(f)=\sum_{t}\psi_{t}(f)y_t$ and $\{x_i,
f\}=\sum_{s}(-1)^{|x_i||\psi_s(f)|}\psi_{s}(f)\{x_i, x_s\}$ for any
$f\in R$, we have that
\begin{equation}
\begin{split}
\psi(\{x_i, f\})&=\sum_{t}\psi_{t}(\{x_i, f\})y_t\\&\equiv
\sum_{t}\psi_{t}(\sum_{s}(-1)^{|x_i||\psi_s(f)|}\psi_{s}(f)f_{is})y_t\\
&\equiv\sum_{s,
t}(-1)^{|x_i||\psi_s(f)|}\psi_{s}(f)\psi_{t}(f_{is})y_t+\sum_{s,
t}(-1)^{|\psi_{s}(f)|(|x_i|+|f_{is}|)}f_{is}\psi_{t}\psi_{s}(f)y_t\\
&\equiv\sum_{s,
t}(-1)^{|x_i||\psi_{s}(f)|}\psi_{s}(f)\psi_{t}(f_{is})y_{t}+\sum_{s,
t}(-1)^{|x_i||\psi_{t}\psi_{s}(f)|}\psi_{t}\psi_{s}(f)f_{it}y_{s}
~mod J
\end{split}\nonumber
\end{equation}
by Lemma \ref{lem+}, where $f_{ij}=\{x_i, x_j\}$. Hence (c) follows from Lemma \ref{lem4}.
\end{proof}

\begin{theorem}\label{th2}
Let $R=T(V)/\langle x_{\alpha}\otimes
x_{\beta}-(-1)^{|x_{\alpha}||x_{\beta}|}x_{\beta}\otimes x_{\alpha}|
~\forall \alpha, ~\beta \in\Lambda \rangle$ be a DG Poisson algebra
with an arbitrary differential $d$ and Poisson structure $\{\cdot, \cdot\}$.
Then the universal enveloping algebra $R^{e}$ has a $\Bbbk$-linear
basis
\begin{center}
$\mathfrak{B}=\{x_1^{i_1}x_2^{i_2}\cdots x_n^{i_n}y_1^{j_1}y_2^{j_2}\cdots y_n^{j_n}|~i_r,
j_r=0, 1, 2\cdots\}$.
 \end{center}
\end{theorem}

\begin{proof}
Since

$$R=\frac{T(V)}{\langle x_{i}\otimes
x_{j}-(-1)^{|x_{i}||x_{j}|}x_{j}\otimes x_{i}| ~\forall i, j
\in\Lambda \rangle}$$
and $\psi(x_ix_j-(-1)^{|x_{i}||x_{j}|}x_jx_i)=0$, the universal
enveloping algebra $R^{e}$ is $R^{e}=F(R)/J'$, where $J'$ is the DG
ideal generated by
\begin{center}
\begin{enumerate}
  \item $~~~~~~~~~~~~x_{ij}: x_ix_j-(-1)^{|x_{i}||x_{j}|}x_jx_i$,
  \item $~~~~~~~~~~~~y_{ij}: y_iy_j-(-1)^{|x_i||x_j|}y_jy_i-\psi(\{x_i, x_j\})$,
  \item $~~~~~~~~~~~~z_{ij}: y_ix_j-(-1)^{|x_i||x_j|}x_jy_i-\{x_i, x_j\}$,
\end{enumerate}
\end{center}
for all $i$, $j$. By Proposition \ref{th1}, it is enough to show that the
generators of $J'$ is closed under composition in $F(R)$. There are
only four possible compositions among the generators of $J'$:
\begin{itemize}
  \item $(x_{ij}, x_{jk})_{x_ix_jx_k}(i>j>k)$
  \item $(y_{ij}, y_{jk})_{y_iy_jy_k}(i>j>k)$
  \item $(y_{ij}, z_{jk})_{y_iy_jx_k}(i>j)$
  \item $(z_{ij}, x_{jk})_{y_ix_jx_k}(j>k)$
\end{itemize}
Case1. $(x_{ij}, x_{jk})_{x_ix_jx_k}(i>j>k)$
\begin{align*}
(x_{ij}, x_{jk})_{x_ix_jx_k}=&x_{ij}x_k-x_ix_{jk}\\
=&(x_ix_j-(-1)^{|x_{i}||x_{j}|}x_jx_i)x_k-x_i(x_jx_k-(-1)^{|x_{j}||x_{k}|}x_kx_j)\\
=&-(-1)^{|x_{i}||x_{j}|}x_jx_ix_k+(-1)^{|x_{j}||x_{k}|}x_ix_kx_j\\
\equiv&-(-1)^{2|x_{i}||x_{j}|}x_ix_jx_k+(-1)^{2|x_{j}||x_{k}|}x_ix_jx_k\equiv 0
~mod (J; x_ix_jx_k).
\end{align*}
Case2. $(y_{ij}, y_{jk})_{y_iy_jy_k}(i>j>k)$

Set $\{x_i, x_j\}=f_{ij}$. Since $\psi(f)=\sum_{s}\psi_{s}(f)y_{s}$
and $\{x_i, f\}=\sum_{s}(-1)^{|x_i||\psi_{s}(f)|}\psi_{s}(f)f_{is}$,
we have
\begin{align*}
(y_{ij}, y_{jk})_{y_iy_jy_k}=&y_{ij}y_k-y_iy_{jk}\\
=&[y_iy_j-(-1)^{|x_i||x_j|}y_jy_i-\psi(f_{ij})]y_k
-y_i[y_jy_k-(-1)^{|x_j||x_k|}y_ky_j-\psi(f_{jk})]\\
\equiv&-(-1)^{|x_i||x_j|}y_j[(-1)^{|x_i||x_k|}y_ky_i+\psi(f_{ik})]-\psi(f_{ij})y_k\\
&+(-1)^{|x_j||x_k|}[(-1)^{|x_i||x_k|}y_ky_i+\psi(f_{ik})]y_j+y_i\psi(f_{jk})\\
\equiv&-(-1)^{|x_i||x_j|+|x_i||x_k|}[(-1)^{|x_j||x_k|}y_ky_j+\psi(f_{jk})]y_i-(-1)^{|x_i||x_j|}y_j\psi(f_{ik})-\psi(f_{ij})y_k\\
&+(-1)^{|x_i||x_k|+|x_j||x_k|}y_k[(-1)^{|x_i||x_j|}y_jy_i+\psi(f_{ij})]+(-1)^{|x_j||x_k|}\psi(f_{ik})y_j+y_i\psi(f_{jk})\\
\equiv&[\sum_{s,
t}(-1)^{|x_i||\psi_{s}(f_{jk})|}\psi_{s}(f_{jk})\psi_{t}(f_{is})y_{t}+\sum_{s,t}(-1)^{|x_i||\psi_{t}\psi_{s}(f_{jk})|}\psi_{t}\psi_{s}(f_{jk})f_{it}y_{s}]\\
&+(-1)^{|x_i||f_{jk}|}[\sum_{s,
t}(-1)^{|x_j||\psi_{s}(f_{ki})|}\psi_{s}(f_{ki})\psi_{t}(f_{js})y_{t}+\sum_{s,t}(-1)^{|x_j||\psi_{t}\psi_{s}(f_{ki})|}\psi_{t}\psi_{s}(f_{ki})f_{jt}y_{s}]\\
&+(-1)^{|x_k||f_{ij}|}[\sum_{s,
t}(-1)^{|x_k||\psi_{s}(f_{ij})|}\psi_{s}(f_{ij})\psi_{t}(f_{ks})y_{t}+\sum_{s,t}(-1)^{|x_k||\psi_{t}\psi_{s}(f_{ij})|}\psi_{t}\psi_{s}(f_{ij})f_{kt}y_{s}]\\
\end{align*}
by Lemma \ref{lem4}. It is easy to see that
\begin{align*}
&\sum_{s, t}(-1)^{|x_k||\psi_{t}\psi_{s}(f_{ij})|}\psi_{t}\psi_{s}(f_{ij})f_{kt}y_{s}
\equiv\sum_{s,
t}(-1)^{|x_s||\psi_{s}(f_{ij})|}f_{ks}\psi_{t}\psi_{s}(f_{ij})y_t ~mod J
\end{align*}
by Lemma \ref{lem+}.

Hence
\begin{align*}
(y_{ij}, y_{jk})_{y_iy_jy_k}\equiv&\sum_{s,
t}(-1)^{|x_i||\psi_{s}(f_{jk})|}[\psi_{s}(f_{jk})\psi_{t}(f_{is})+(-1)^{|\psi_{s}(f_{jk})||f_{is}|}f_{is}\psi_{t}\psi_{s}(f_{jk})]y_t\\
&+\sum_{s,
t}(-1)^{|x_k||f_{ij}|+|x_k||\psi_{s}(f_{ij})|}[\psi_{s}(f_{ij})\psi_{t}(f_{ks})+(-1)^{|\psi_{s}(f_{ij})||f_{ks}|}f_{ks}\psi_{t}\psi_{s}(f_{ij})]y_t\\
&+\sum_{s,
t}(-1)^{|x_i||f_{jk}|+|x_j||\psi_{s}(f_{ki})|}[\psi_{s}(f_{ki})\psi_{t}(f_{js})+(-1)^{|\psi_{s}(f_{ki})||f_{js}|}f_{js}\psi_{t}\psi_{s}(f_{ki})]y_t\\
\equiv& ~0 ~mod (J; y_iy_jy_k)
\end{align*}
by Lemma \ref{lem5}.

Case3. $(y_{ij}, z_{jk})_{y_iy_jx_k}(i>j)$
\begin{align*}
(y_{ij}, z_{jk})_{y_iy_jx_k}=&y_{ij}x_k-y_iz_{jk}\\
=&(y_iy_j-(-1)^{|x_i||x_j|}y_jy_i-\psi(\{x_i, x_j\}))x_k-y_i(
y_jx_k-(-1)^{|x_j||x_k|}x_ky_j-\{x_j, x_k\})\\
\equiv&-(-1)^{|x_i||x_j|}y_j((-1)^{|x_i||x_k|}x_ky_i+\{x_i, x_k\})+(-1)^{|x_j||x_k|}((-1)^{|x_i||x_k|}x_ky_i+\{x_i, x_k\})y_j\\
&-\psi(\{x_i, x_j\})x_k+y_i\{x_j, x_k\}\\
\equiv&-(-1)^{|x_i||x_j|+|x_i||x_k|}((-1)^{|x_j||x_k|}x_ky_j+\{x_j,
x_k\})y_i-(-1)^{|x_i||x_j|}y_j\{x_i, x_k\}-\psi(\{x_i, x_j\})x_k\\
&+(-1)^{|x_j||x_k|+|x_i||x_k|}x_k((-1)^{|x_i||x_j|}y_jy_i+\psi(\{x_i,
x_j\}))+(-1)^{|x_j||x_k|}\{x_i, x_k\}y_j+y_i\{x_j, x_k\}\\
\equiv&-(-1)^{|x_i||x_j|+|x_i||x_k|}\{x_j,
x_k\}y_i-(-1)^{|x_i||x_j|}((-1)^{|x_j||\{x_i, x_k\}|}\{x_i, x_k\}y_j+\{x_j, \{x_i, x_k\}\})\\
&-((-1)^{|x_k||\{x_i, x_j\}|}x_k\psi(\{x_i, x_j\})+\{\{x_i, x_j\}, x_k\})+(-1)^{|x_j||x_k|}\{x_i, x_k\}y_j\\
&+(-1)^{|x_j||x_k|+|x_i||x_k|}x_k\psi(\{x_i,
x_j\})+(-1)^{|x_i||\{x_j, x_k\}|}\{x_j, x_k\}y_i+\{x_i, \{x_j, x_k\}\}\\
\equiv&-(-1)^{|x_i||x_j|}\{x_j, \{x_i, x_k\}\}-\{\{x_i, x_j\},
x_k\}+\{x_i, \{x_j, x_k\}\}\\
\equiv& ~0 ~mod (J; y_iy_jx_k)
\end{align*}
by the graded Jacobi identity, Lemmas \ref{lem2} and \ref{lem3} .

Case4. $(z_{ij}, x_{jk})_{y_ix_jx_k}(j>k)$
\begin{align*}
(z_{ij}, x_{jk})_{y_ix_jx_k}=&z_{ij}x_k-y_ix_{jk}\\
=&(y_ix_j-(-1)^{|x_i||x_j|}x_jy_i-\{x_i, x_j\})x_k-y_i(
x_jx_k-(-1)^{|x_j||x_k|}x_kx_j)\\
\equiv&-(-1)^{|x_i||x_j|}x_j((-1)^{|x_i||x_k|}x_ky_i+\{x_i,
x_k\})\\
&-\{x_i, x_j\}x_k+(-1)^{|x_j||x_k|}((-1)^{|x_i||x_k|}x_ky_i+\{x_i, x_k\})x_j\\
\equiv&-(-1)^{|x_i||x_j|+|x_i||x_k|+|x_j||x_k|}x_kx_jy_i-(-1)^{|x_i||x_j|}x_j\{x_i,
x_k\}-\{x_i, x_j\}x_k\\
&+(-1)^{|x_j||x_k|+|x_i||x_k|}x_k((-1)^{|x_i||x_j|}x_jy_i+\{x_i, x_j\})+(-1)^{|x_j||x_k|}\{x_i, x_k\}x_j\\
\equiv& ~0 ~mod (J; y_iy_jx_k).
\end{align*}
\end{proof}

\begin{prop}\label{prop1}
Let $R=T(V)/\langle x_{\alpha}\otimes
x_{\beta}-(-1)^{|x_{\alpha}||x_{\beta}|}x_{\beta}\otimes x_{\alpha}|
~\forall \alpha, ~\beta \in\Lambda \rangle$ be a DG Poisson algebra
with an arbitrary differential $d$ and Poisson structure $\{\cdot, \cdot\}$.
Suppose that $A=R/I$ is a DG Poisson homomorphic image of $R$, where
$I$ is a subset of monic elements of $R$. If $I\cup \psi (I)$ is
closed under composition in $R^{e}$, then
\begin{center}
$\{u\in\mathfrak{B}|u$ is $I\cup \psi(I)$-standard$\}$
 \end{center}
forms a $\Bbbk$-linear basis of $A^{e}$, where $\mathfrak{B}$ is the
one given in Theorem \ref{th2}.
\end{prop}

\begin{proof}
It follows immediately from Proposition \ref{th1} and Theorem \ref{th2}.
\end{proof}

\begin{remark}
Although the degree of the Poisson bracket is zero in our definition
of DG Poisson algebra, the result obtained in this paper are
also true for DG Poisson algebras of degree $n$ with some expected
signs, where $n\in \mathbb{Z}$ is the degree of the Poisson bracket.
\end{remark}

\begin{exa}
Let $$A=\frac{\Bbbk<x_{1}, x_{2}>}{(x_{1}x_{2}, ~x_{2}x_{1}, ~x_{2}^2)},$$
where $|x_{1}=2, |x_{2}|=3$.
Let $d: A\rightarrow A$ be a $\Bbbk$-linear map of degree 1 by
\begin{center}
$d(x_{1})=x_{2}, ~~~d(x_{2})=0$.
 \end{center}
Moreover, we can define the Poisson bracket by
\begin{center}
$\{x_{1}, x_{2}\}=-\{x_{2}, x_{1}\}=x_{2}^2, ~~~\{x_{1}, x_{1}\}=\{x_{2}, x_{2}\}=0$.
 \end{center}
Note that $0=x_2^2\in A$, it is easy to see that $A$ is a DG Poisson algebra of degree 1, where the degree of the Poisson bracket is 1.
Here, we can suppose that
\begin{center}
$R=\Bbbk<x_{1}, x_{2}>/(x_{1}x_{2}-x_{2}x_{1}, ~x_{2}^2)$
 \end{center}
and  $I=<x_1x_{2}>$ is a DG Poisson ideal of $R$. Given an ordering on the set of generators of $R^e$ by
\begin{center}
$x_{1}<x_{2}<y_{1}<y_{2}$.
 \end{center}
Since the set consisting of $x_1x_{2}, \psi(x_1x_{2})=x_1y_2+x_2y_1$
is closed under composition in $R^e $, the universal enveloping algebra $A^e$ has a $\Bbbk$-linear basis
\begin{center}
$\{x_{1}^{i_1}x_{2}^{i_2}y_{1}^{j_1}y_{2}^{j_2} | ~i_1i_{2}=0, i_2j_1=0\}$
 \end{center}
by Proposition \ref{prop1}.
\end{exa}

\begin{corollary}\label{cor1}
Let $R=T(V)/\langle x_{\alpha}\otimes
x_{\beta}-(-1)^{|x_{\alpha}||x_{\beta}|}x_{\beta}\otimes x_{\alpha}|
~\forall \alpha, ~\beta \in\Lambda \rangle$ be a DG Poisson algebra
with an arbitrary differential $d$ and Poisson structure $\{\cdot, \cdot\}$.
Then the universal enveloping algebra $R^{e}$ is a DG free left and
right $R$-module with basis
\begin{center}
$\mathfrak{C}=\{y_1^{j_1}y_2^{j_2}\cdots y_n^{j_n}|~j_r=0, 1, 2\cdots\}$.
 \end{center}
\end{corollary}

\begin{proof}
By Theorem \ref{th2}, $R^{e}$ is a DG free left $R$-module. To show that
$R^{e}$ is a DG free right $R$-module, it is enough to prove that
\begin{center}
$\mathfrak{B'}=\{y_1^{j_1}y_2^{j_2}\cdots y_n^{j_n}x_1^{i_1}x_2^{i_2}\cdots\cdot x_n^{i_n}|~j_r, i_r=0, 1, 2\cdots\}$
\end{center}
forms a $\Bbbk$-linear basis for $R^{e}$.

We show that $\mathfrak{B'}$ is $\Bbbk$- linearly independent. Let
\begin{equation}
\begin{split}
&\sum_{1\leq i_1+\cdots+i_n}\alpha_{i_1\cdots i_n}y_1^{i_1}\cdots y_n^{i_n}+\sum_{1\leq
j_1+\cdots+j_n, 1\leq
l_1+\cdots+l_n}\beta_{j_1\cdots j_nl_1\cdots l_n}y_1^{j_1}\cdots y_n^{j_n}x_1^{l_1}\cdots x_n^{l_n}\\
&+\sum_{1\leq m_1+\cdots+m_n}\gamma_{m_1\cdots m_n}x_1^{m_1}\cdots x_n^{m_n}+\delta1=0,\\
\end{split}
\end{equation}
where $\alpha_{i_1\cdots i_n}, \beta_{j_1\cdots j_nl_1\cdots l_n}, \gamma_{m_1\cdots m_n}$ and $\delta\in \Bbbk$. If there exist nonzero $\beta_{j_1\cdots j_nl_1\cdots l_n}$'s in the second term of formula (3.1),
assume that
\begin{center}
$z=y_1^{j_1}\cdots y_n^{j_n}x_1^{l_1}\cdots x_n^{l_n}$
 \end{center}
 is maximal among the monomials with nonzero coefficients in the
 second term of formula (3.1). Using (a) of Lemma \ref{lem7}, all monomials in the second term of formula (3.1) can be
expressed as $\Bbbk$-linear combinations of monomials in
$\mathfrak{B}$ of Theorem \ref{th2}. Then the coefficient of $z$ is the
coefficient of $x_1^{l_1}\cdots x_n^{l_n}y_1^{j_1}\cdots y_n^{j_n}$. Thus all
$\beta_{j_1\cdots j_nl_1\cdots l_n}$'s in the
second term of formula (3.1) are zero and hence all
$\alpha_{i_1\cdots i_n}, \gamma_{m_1\cdots m_n}$
and $\delta$ in formula (3.1) are also zero by Theorem \ref{th2}. By
Lemma \ref{lem7}, it is easy to see that all monomials in $\mathfrak{B}$
of Theorem \ref{th2} are spanned by $\mathfrak{B'}$. Therefore
$\mathfrak{B'}$ forms a $\Bbbk$-linear basis of $R^{e}$, as
required.
\end{proof}

\medskip
\section{Simple DG Poisson modules}
In this section, for a given DG Poisson algebra $A=R/I$, where
$R$ and $A$ are defined as in Section 3, we prove that a DG symplectic ideal $P$ of
$A$ is the annihilator of a simple DG Poisson $A$-module.

\begin{lemma}\label{lem8}
Let $\mathcal{A}$ be the set of all DG ideals in a DG Poisson
algebra $R$ and let $\mathcal{B}$ be the set of all left DG ideals
of $R^{e}$. For $I\in \mathcal{A}$,
let $\overline{I}$ be the left DG ideal of $R^{e}$ generated by $I$.
Then the map $I\rightarrow \overline{I}$ is an injective map from
$\mathcal{A}$ into $\mathcal{B}$ and preserves inclusion, that is,
$\overline{I}\cap R=I$. In particular, $I$ is a DG Poisson ideal of
$R$ if and only if $\overline{I}$ is an DG ideal of $R^{e}$.
\end{lemma}

\begin{proof}
Let $I\in \mathcal{A}$. Since $R^{e}$ is a DG free right
$R$-module with basis $\mathfrak{C}$ given in Corollary \ref{cor1}, every
element of $\overline{I}=R^{e}I$ is of the form
\begin{center}
$1a_0+Y_1a_1+Y_2a_2+\cdots+Y_sa_s$
\end{center}
for some $a_i\in I, 1\neq Y_i\in \mathfrak{C}$. Hence, if
$I\subseteq J$ then $\overline{I}\subseteq \overline{J}$ and
$I=\overline{I}\cap R$ by Corollary \ref{cor1}, thus the map $I\rightarrow
\overline{I}$ is injective. The second statement follows immediately
from Lemma \ref{lem7}.
\end{proof}

\begin{remark}
The map $I\rightarrow \overline{I}$ from $\mathcal{A}$ into
$\mathcal{B}$ is not surjective. In \cite{OPS}, we can see that  in
the situation of the DG Poisson algebra concentrated in degree 0
with trivial differential, this map is not a surjective map.
\end{remark}

\begin{lemma}\label{lem6}
Let $(R^{e}, m_R, h_R)$ be the universal enveloping algebra of a DG
Poisson algebra $(R, d, \{\cdot, \cdot\})$. If
$I$ is a DG Poisson ideal of $R$ and $Q$ is the DG ideal of $R^{e}$
generated by $I$ and $h_R(I)$, then $(R^{e}/Q, m', h')$ is the
universal enveloping algebra of $A=R/I$, where
\begin{center}
$m': A\rightarrow R^{e}/Q, ~~~~~~~~m'(r+I)=m_R(r)+Q$,
\end{center}
\begin{center}
$h': A\rightarrow R^{e}/Q, ~~~~~~~~h'(r+I)=h_R(r)+Q$.
\end{center}
\end{lemma}

\begin{proof}
Since  $Q$ is the DG ideal of $R^{e}$ generated by $I$ and $h_R(I)$,
we have
\begin{center}
$Q=R^{e}I+R^{e}h_R(I)R^{e}$
\end{center}
by Lemma \ref{lem8}. Let $(D, \delta)\in \textbf{DGA}$ with a DG algebra
map $f: (A, d)\rightarrow (D, \delta)$ and a DG Lie algebra map
$g: (A, \{\cdot, \cdot\}_A, d)\rightarrow (D_P, [\cdot, \cdot], \delta)$ satisfying
\begin{align*}
f(\{a, b\})&=g(a)f(b)-(-1)^{|a||b|}f(b)g(a),\\
g(ab)&=f(a)g(b)+(-1)^{|a||b|}f(b)g(a),
\end{align*}
for any homogeneous elements $a, b\in A$. Denote by $\pi'$ the
canonical projection from $R$ onto $A$, then there exists a unique
DG algebra map $\phi$ from $R^{e}$ into $D$ such that $\phi
m_R=f\pi'$ and $\phi h_R=g\pi'$ since $(R^{e}, m_R, h_R)$ is the
universal enveloping algebra of a DG Poisson algebra $(R, \{\cdot,
\cdot\}, d)$. Hence
\begin{center}
$\phi(I)=\phi m_R(I)=f\pi'(I)=0$ \quad and \quad $\phi h_R(I)=g\pi'(I)=0$,
\end{center}
so the DG ideal $Q$ is contained in the kernel of $\phi$. Thus
$\phi$ induces the DG algebra map $\phi'$ from $R^{e}/Q$ into $D$
such that $\phi'(u+Q)=\phi(u)$ for all $u\in R^{e}$. Clearly, for
all $r+I\in A$, we have that
\begin{equation}
\begin{split}
&\phi'm'(r+I)=\phi'(m_R(r)+Q)=\phi(m_R(r))=f\pi'(r)=f(r+I),\\
&\phi'h'(r+I)=\phi'(h_R(r)+Q)=\phi(h_R(r))=g\pi'(r)=g(r+I).
\end{split}\nonumber
\end{equation}
If $\phi_1$ is an DG algebra map from $R^{e}/Q$ into $D$ such that
$\phi_1m'=f$ and $\phi_1h'=g$, then we have $\phi_1=\phi'$ since
$R^{e}/Q$ is generated by $m'(A)$ and $h'(A)$. It completes the
proof.
\end{proof}

\begin{remark}\label{rmk1}
From the Lemma \ref{lem6}, we have the following two observations:
\begin{itemize}
  \item Since $R^{e}$ is the universal enveloping algebra of a DG Poisson
algebra $R$, there is an DG algebra homomorphism $\eta$ from $R^{e}$
into $A^e$ such that $\eta m_R=m_A\pi'$ and $\eta h_R=h_A\pi'$.
\begin{eqnarray*}
\xymatrix{
  R \ar[d]_{m_R, ~h_R}   \ar[rr]^{\pi'}
         &\       & A \ar[d]^{m_A, ~h_A}  \\
  R^{e}  \ar[rr]_{\eta}
         &  \     & A^e             }
\end{eqnarray*}
  \item $\eta$ is an epimorphism, since $A^e$
is generated by $A=m_A(A)$ and $h_A(A)$. Set
\begin{center}
$I=ker \pi', ~~~~Q=R^{e}I+R^{e}h_R(I)R^{e}$.
\end{center}
We have $ker(\eta)=Q$.
\end{itemize}
\end{remark}

\begin{lemma}\label{lem9}
If $f$ and $g$ are DG $\Bbbk$-linear maps from a DG Poisson algebra
$A$ into a DG algebra $B$ such that
\begin{align*}
f(\{a, b\})&=g(a)f(b)-(-1)^{|a||b|}f(b)g(a),\\
g(ab)&=f(a)g(b)+(-1)^{|a||b|}f(b)g(a),
\end{align*}
for any homogeneous elements $a, b\in A$, then
\begin{align*}
f(\{a, b\})&=f(a)g(b)-(-1)^{|a||b|}g(b)f(a),\\
g(ab)&=g(a)f(b)+(-1)^{|a||b|}g(b)f(a).
\end{align*}
\end{lemma}

\begin{proof}
For any homogeneous elements $a, b\in A$, we have
\begin{align*}
&f(\{a, b\})+g(ab)=g(a)f(b)+f(a)g(b),\\
&f(\{b, a\})+g(ba)=g(b)f(a)+f(b)g(a).
\end{align*}
Note that
\begin{center}
$\{a, b\}=-(-1)^{|a||b|}\{b, a\}, ~~ab=(-1)^{|a||b|}ba$.
\end{center}
 Hence
\begin{align*}
2(-1)^{|a||b|}f(\{a,
b\})&=(-1)^{|a||b|}g(a)f(b)+(-1)^{|a||b|}f(a)g(b)-g(b)f(a)-f(b)g(a)\\
&=(-1)^{|a||b|}f(a)g(b)-g(b)f(a)+(-1)^{|a||b|}(-1)^{|a||b|}f(\{a,
b\})
\end{align*}
and
\begin{align*}
2(-1)^{|a||b|}g(ab)&=(-1)^{|a||b|}g(a)f(b)+(-1)^{|a||b|}f(a)g(b)+g(b)f(a)+f(b)g(a)\\
&=(-1)^{|a||b|}g(a)f(b)+g(b)f(a)+(-1)^{|a||b|}g(ab).
\end{align*}
Therefore, we have the conclusion.
\end{proof}

\begin{lemma}\label{lem10}
Let $(R^{e}, m_R, h_R)$ be the universal enveloping algebra of a DG
Poisson algebra $R$. Then
\begin{equation}
h_R(x_i^{n})=
   \begin{cases}
   (k+k(-1)^{(2k-1)|x_i|^2})y_ix_i^{2k-1}, &if~n=2k,\\
   (k+1+k(-1)^{(2k-1)|x_i|^2})y_ix_i^{2k}, &if~n=2k+1.\\
   \end{cases}
\end{equation}
\end{lemma}

\begin{proof}
Since
\begin{center}
$h_R(ab)=h_R(a)m_R(b)+(-1)^{|a||b|}h_R(b)m_R(a)=h_R(a)b+(-1)^{|a||b|}h_R(b)a$
\end{center}
for all elements $a, b\in R$ by Lemma \ref{lem9}. Observe that formula
(4.1) is true on $n=1$ and $n=2$ since
\begin{equation*}
   \begin{cases}
   h_R(x_i)=y_i,\\
   h_R(x_i^2)=h_R(x_i)x_i+(-1)^{|x_i||x_i|}h_R(x_i)x_i=(1+(-1)^{|x_i|^2})y_ix_i.
   \end{cases}
\end{equation*}
Set
\begin{equation}
\bigtriangleup (n)=
   \begin{cases}
   k+k(-1)^{(2k-1)|x_i|^2}), &if~n=2k,\\
   k+1+k(-1)^{(2k-1)|x_i|^2}, &if~n=2k+1.\nonumber\\
   \end{cases}
\end{equation}
Thus in order to complete the proof, we only
need to show
\begin{center}
$h_R(x_i^{n+1})=\bigtriangleup (n+1)y_ix_i^{n}$
\end{center} provided that $h_R(x_i^{n})=\bigtriangleup (n)y_ix_i^{n-1}$.

Indeed if $n=2k$, then
\begin{equation}
\begin{split}
h_R(x_i^{2k+1})&=h_R(x_i^{2k})x_i+(-1)^{|x_i^{2k}||x_i|}h_R(x_i)x_i^{2k}\\
&=\bigtriangleup (2k)y_ix_i^{2k-1}x_i+(-1)^{|x_i^{2k}||x_i|}y_ix_i^{2k}\\
&=(k+1+k(-1)^{(2k-1)|x_i|^2})y_ix_i^{2k}\\&=\bigtriangleup (n+1)y_ix_i^{n}.\\
\end{split}\nonumber
\end{equation}
On the other hand, if $n=2k+1$, then
\begin{equation}
\begin{split}
h_R(x_i^{2k+2})&=h_R(x_i^{2k+1})x_i+(-1)^{|x_i^{2k+1}||x_i|}h_R(x_i)x_i^{2k+1}\\
&=\bigtriangleup (2k+1)y_ix_i^{2k}x_i+(-1)^{|x_i^{2k+1}||x_i|}y_ix_i^{2k+1}\\
&=(k+1+(k+1)(-1)^{(2k+1)|x_i|^2})y_ix_i^{2k+1}\\&=\bigtriangleup (n+1)y_ix_i^{n},\\
\end{split}\nonumber
\end{equation}
as required.
\end{proof}

Using the above notation, we have the following lemma.
\begin{lemma}\label{lem11}
Let $(R^{e}, m_R, h_R)$ be the universal enveloping algebra of a DG
Poisson algebra $R$. Then
\begin{equation}
\begin{split}
h_R(x_1^{i_1}x_2^{i_2}\cdots x_n^{i_n})=\sum_{k=1}^n(-1)^{|x_1^{i_1}\cdots x_{k-1}^{i_{k-1}}||x_k|}\bigtriangleup(i_k)y_kx_1^{i_1}\cdots x_{k-1}^{i_{k-1}}x_k^{i_k-1}x_{k+1}^{i_{k+1}}\cdots x_n^{i_n}.
\end{split}
\end{equation}
\end{lemma}

\begin{proof}
For each monomial $x_1^{i_1}x_2^{i_2}\cdots x_n^{i_n}$ of $R$, we proceed the
proof using induction on the number of indeterminates. By
Lemma \ref{lem10}, it is clear that formula (4.2) is true on the
$x_1^{i_1}$. Thus in order to complete the proof, we only need to
show the formula (4.2) on the $x_1^{i_1}x_2^{i_2}\cdots x_n^{i_n}$ provided
that
\begin{equation}
\begin{split}
h_R(x_1^{i_1}x_2^{i_2}\cdots x_{n-1}^{i_{n-1}})=\sum_{k=1}^{n-1}(-1)^{|x_1^{i_1}\cdots x_{k-1}^{i_{k-1}}||x_k|}\bigtriangleup(i_k)y_kx_1^{i_1}\cdots x_{k-1}^{i_{k-1}}x_k^{i_k-1}x_{k+1}^{i_{k+1}}\cdots x_{n-1}^{i_{n-1}}.
\end{split}\nonumber
\end{equation}
Since
\begin{center}
$h_R(ab)=h_R(a)m_R(b)+(-1)^{|a||b|}h_R(b)m_R(a)=h_R(a)b+(-1)^{|a||b|}h_R(b)a$
\end{center}
for all elements $a, b\in R$ by Lemma \ref{lem9}, we have
\begin{equation}
\begin{split}
h_R(x_1^{i_1}x_2^{i_2}\cdots x_n^{i_n})=&h_R(x_1^{i_1}x_2^{i_2}\cdots x_{n-1}^{i_{n-1}})x_n^{i_n}
+(-1)^{|x_1^{i_1}x_2^{i_2}\cdots x_{n-1}^{i_{n-1}}||x_n^{i_n}|}h_R(x_n^{i_n})x_1^{i_1}x_2^{i_2}\cdots x_{n-1}^{i_{n-1}}\\
=&(\sum_{k=1}^{n-1}(-1)^{|x_1^{i_1}\cdots x_{k-1}^{i_{k-1}}||x_k|}\bigtriangleup(i_k)y_kx_1^{i_1}\cdots
x_{k-1}^{i_{k-1}}x_k^{i_k-1}x_{k+1}^{i_{k+1}}\cdots x_{n-1}^{i_{n-1}})x_n^{i_n}\\
&+(-1)^{|x_1^{i_1}x_2^{i_2}\cdots x_{n-1}^{i_{n-1}}||x_n^{i_n}|}\bigtriangleup
(i_n)y_nx_n^{i_n-1}x_1^{i_1}x_2^{i_2}\cdots x_{n-1}^{i_{n-1}}\\
=&\sum_{k=1}^{n-1}(-1)^{|x_1^{i_1}\cdots x_{k-1}^{i_{k-1}}||x_k|}\bigtriangleup(i_k)y_kx_1^{i_1}\cdots
x_{k-1}^{i_{k-1}}x_k^{i_k-1}x_{k+1}^{i_{k+1}}\cdots x_{n-1}^{i_{n-1}}x_n^{i_n}\\
&+(-1)^{|x_1^{i_1}x_2^{i_2}\cdots x_{n-1}^{i_{n-1}}||x_n|}\bigtriangleup
(i_n)y_nx_1^{i_1}x_2^{i_2}\cdots x_{n-1}^{i_{n-1}}x_n^{i_n-1}\\
=&\sum_{k=1}^{n}(-1)^{|x_1^{i_1}\cdots x_{k-1}^{i_{k-1}}||x_k|}\bigtriangleup(i_k)y_kx_1^{i_1}\cdots
x_{k-1}^{i_{k-1}}x_k^{i_k-1}x_{k+1}^{i_{k+1}}\cdots x_{n-1}^{i_{n-1}}x_n^{i_n},\\
\end{split}\nonumber
\end{equation}
as required.
\end{proof}

\begin{prop}\label{th3}
Let $(R^{e}, m_R, h_R)$ be the universal enveloping algebra of a DG
Poisson algebra $R$. Then, for every
monomial  $x_1^{i_1}x_2^{i_2}\cdots x_n^{i_n}\in R$, the expression of
$z=y_1^{j_1}y_2^{j_2}\cdots y_n^{j_n}h_R(x_1^{i_1}x_2^{i_2}\cdots x_n^{i_n})$
as a right $R$-combination of monomial in $\mathfrak{C}$ of
Corollary \ref{cor1} is of the form
\begin{center}
$k_1Y_1a_1+k_2Y_2a_2+\cdots+k_sY_sa_s, ~k_i\in \Bbbk, ~1\neq
Y_i\in\mathfrak{C}$ and $a_i\in R$.
\end{center}
That is, the coefficient of 1 under the expression of $z$ as a right
$R$-combination of monomials in $\mathfrak{C}$ is zero.
\end{prop}

\begin{proof}
Since
\begin{center}
$h_R(ab)=m_R(a)h_R(b)+(-1)^{|a||b|}m_R(b)h_R(a)=ah_R(b)+(-1)^{|a||b|}bh_R(a)$,
\end{center}
for all elements $a, b\in R$, we have $h_R(1)=0$, that is,
$h_R(k)=0$ for every $k\in \Bbbk$. Express
$h_R(x_1^{i_1}x_2^{i_2}\cdots x_n^{i_n})$ as given in Lemma \ref{lem11} and
then it suffices to prove the statement for the case
$z=y_1^{j_1}y_2^{j_2}\cdots y_n^{j_n}y_i$.

For a monomial $y_1^{j_1}y_2^{j_2}\cdots y_n^{j_n}$, we proceed the
proof using induction on the $r=j_1+j_2+\cdots+j_n$. If $r=0$ or $1$, then it
is trivial. Assume that the statement is true for monomials with
length less than $j_1+j_2+\cdots+j_n$, we should proof the statement is
also true for monomials with length equal $j_1+j_2+\cdots+j_n$. There are two possibilities: (\uppercase\expandafter{\romannumeral1}) $j_n\neq 0$, (\uppercase\expandafter{\romannumeral2}) $j_n=0$ (without loss of generality, we can further assume that $j_{n-1}\neq 0$).

Case \uppercase\expandafter{\romannumeral1}. $j_n\neq 0$: Note that
\begin{center}
$y_1^{j_1}y_2^{j_2}\cdots y_n^{j_n}y_i=y_1^{j_1}y_2^{j_2}\cdots y_n^{j_n-1}((-1)^{|x_n||x_i|}y_iy_n+h_R(\{x_n,
x_i\}))$
\end{center}
by Lemma \ref{lem7}. By the induction hypothesis, we have the expression of
$y_1^{j_1}y_2^{j_2}\cdots y_n^{j_n-1}y_i$ as a right $R$-combination of
monomial in $\mathfrak{C}$ of Corollary \ref{cor1} is of the form
\begin{center}
$k_1Y_1a_1+k_2Y_2a_2+\cdots+k_sY_sa_s, ~k_i\in \Bbbk, ~1\neq
Y_i\in\mathfrak{C}$ and $a_i\in R$.
\end{center}
Then the result on the case
\begin{center}
$(-1)^{|x_n||x_i|}y_1^{j_1}y_2^{j_2}\cdots y_n^{j_n-1}y_iy_n$
\end{center}
is true by $(a)$ of Lemma \ref{lem7}. On the other hand, since $\{x_n,
x_i\}$ can expressed as polynomial in $R$, we can see that the result on
the case
\begin{center}
$y_1^{j_1}y_2^{j_2}\cdots y_n^{j_n-1}h_R(\{x_n, x_i\})$
\end{center}
is also true by Lemma \ref{lem11} and the induction hypothesis.

Case \uppercase\expandafter{\romannumeral2}. $j_n=0$: Note that
\begin{center}
$y_1^{j_1}y_2^{j_2}\cdots y_{n-1}^{j_{n-1}}y_i=y_1^{j_1}y_2^{j_2}\cdots y_{n-1}^{j_{n-1}-1}((-1)^{|x_{n-1}||x_i|}y_iy_{n-1}+h_R(\{x_{n-1},
x_i\}))$.
\end{center}
Similarly, we have
\begin{align*}
y_1^{j_1}y_2^{j_2}\cdots y_{n-1}^{j_{n-1}-1}y_iy_{n-1}&=(k_1'Y_1'a_1'+k_2'Y_2'a_2'+\cdots+k_s'Y_s'a_s')y_{n-1}\\&=\sum_{i=1}^s(-1)^{|x_{n-1}||a_i'|}k_i'Y_i'(y_{n-1}a_i'-
\{x_{n-1}, a_i'\})
\end{align*}
by Lemma \ref{lem7} and the induction hypothesis, where $~k_i'\in \Bbbk, ~1\neq
Y_i'\in\mathfrak{C}$ and $a_i'\in R$.
Thus, the result on the case
\begin{center}
$\sum_i(-1)^{|x_{n-1}||a_i'|}k_i'Y_i'y_{n-1}a_i'$
\end{center}
is true by the case \uppercase\expandafter{\romannumeral1}. Further, similar to the proof of the case \uppercase\expandafter{\romannumeral1}, it is easy to see that the result on the case \uppercase\expandafter{\romannumeral2} is also true.

Therefore, we finish the proof.
\end{proof}

\begin{lemma}\label{lem12}
Let $(R^{e}, m_R, h_R)$ be the universal enveloping algebra of a DG
Poisson algebra $R$. If $M\neq R$ is
a DG ideal of $R$ and $Q$ is a DG Poisson ideal contained in $M$,
then
\begin{center}
$R^{e}M+R^{e}h_R(Q)R^{e}\neq R^{e}.$
\end{center}
\end{lemma}

\begin{proof}
Since $R^{e}$ is a DG free right $R$-module by Corollary \ref{cor1}, every
element of $R^{e}M+R^{e}h_R(Q)R^{e}$ is of the form
\begin{center}
$\sum_{i}b_im_i+c_ih_R(q_i)a_i, ~a_i, b_i, c_i\in R^{e}, ~m_i\in M,
~q_i\in Q.$
\end{center}
Express each $a_i$ as a right $R$-combination of monomials $Y_j$ in
$\mathfrak{C}$ of Corollary \ref{cor1}. Since
\begin{center}
$y_ih_R(q_i)=(-1)^{|x_i||q_i|}h_R(q_i)y_i+h_R(\{x_i, q_i\})$
\end{center}
for each $i=1,2, \cdots, n$ by $(c)$ of Lemma \ref{lem7} and since $\{x_i,
q_i\}\in Q$, we may assume that $a_i\in R$ for all $i$. Next express
each $c_i$ as a right $R$-combination of monomials $Y_j$ in
$\mathfrak{C}$ of Corollary \ref{cor1}. For $r\in R$, since
\begin{center}
$h_R(q_i)r=(-1)^{|r||q_i|}rh_R(q_i)+\{q_i, r\}$
\end{center}
by $(b)$ of Lemma \ref{lem7} and $\{q_i, r\}\in Q\subseteq M$, we may
assume that each $c_i$ is a $\Bbbk$-linear combination of monomials
in $\mathfrak{C}$ of Corollary \ref{cor1}. Thus by Proposition \ref{th3},
$\sum_{i}c_ih_R(q_i)a_i$ can be written as
\begin{center}
$\sum_{i}c_ih_R(q_i)a_i=\sum_{i}k_i'Y_i'r_i', ~k_i'\in \Bbbk, ~1\neq
Y_i'\in\mathfrak{C}$ and $r_i'\in R.$
\end{center}
Finally, express each $b_i$ as a right $R$-combination of monomials
$Y_j$ in $\mathfrak{C}$ of Corollary \ref{cor1}. Then every element of
$R^{e}M+R^{e}h_R(Q)R^{e}$ is of the form
\begin{center}
$1m_0+k_1Y_1r_1+k_2Y_2r_2+...+k_sY_sr_s, ~m_0\in M, ~k_i\in \Bbbk, ~1\neq Y_i\in\mathfrak{C}$ and $r_i\in R.$
\end{center}
Thus $R^{e}M+R^{e}h_R(Q)R^{e}$ does not contain the unity since $M\neq R$ is a DG ideal of $R$. Therefore, we complete the proof.
\end{proof}

Now, we study the simple DG Poisson module.

Let $(B, \cdot, \{\cdot, \cdot\}, d)\in \textbf{DGPA}$ and let $(B^{e}, m_B,
h_B)$ be the universal enveloping algebra of $B$, the definition of
a DG Poisson module $M$ over a DG Poisson algebra $B$ is given in
Definition \ref{defn1}. Note that the annihilator of a DG Poisson $B$-module
$M$ is defined to be
\begin{center}
$ann_B(M)=\{b\in B | ~b\cdot M=0\}$,
\end{center}
which is a DG Poisson ideal of $B$. Since $M$ is a left DG
$B^{e}$-module, it is easy to prove that
\begin{center}
$ann_B(M)=m_B^{-1}(ann_{B^{e}}(M))=ann_{B^{e}}(M)\cap B$,
\end{center}
where $B$ is a finitely generated DG Poisson algebra.

Now follow the idea of \cite{OPS}, we can define a DG symplectic
ideal. A DG Poisson ideal $Q$ of $B$ is said to be DG symplectic
ideal if there is a maximal DG ideal $M$ of $(B, \cdot, \{\cdot, \cdot\}, d)$
such that $Q$ is the largest DG Poisson ideal contained in $M$. Note
that a simple DG Poisson module over a DG Poisson algebra $B$ is the
DG Poisson module over $B$ that has no non-zero proper DG Poisson
submodule.

\begin{theorem}
Let $R=T(V)/\langle x_{\alpha}\otimes
x_{\beta}-(-1)^{|x_{\alpha}||x_{\beta}|}x_{\beta}\otimes x_{\alpha}|
~\forall \alpha, ~\beta \in\Lambda \rangle$ be a DG Poisson algebra
with an arbitrary differential $d$ and Poisson structure $\{\cdot, \cdot\}$.
Suppose that $A=R/I$ is a DG Poisson homomorphic image of $R$. If $P$ is a DG symplectic ideal of $A$, then $P$ is the
annihilator of a simple DG Poisson $A$-module.
\end{theorem}

\begin{proof}
From the above, let $\pi'$ be the canonical projection from $R$ onto
$A$ and let $M$ be a maximal DG ideal of $A$ such that $P$ is the
largest DG Poisson ideal contained in $M$. By Remark \ref{rmk1}, there is
an epimorphism $\eta$ from $R^{e}$ onto $A^e$ such that
$ker(\eta)=R^{e}I+R^{e}h_R(I)R^{e}$.
\begin{eqnarray*}
\xymatrix{
  R \ar[d]_{m_R, ~h_R} \ar[rr]^{\pi'}
          &\      & A \ar[d]^{m_A, ~h_A}  \\
  R^{e} \ar[rr]_{\eta}
           &\     & A^e             }
\end{eqnarray*}
Note that $\pi'^{-1}(M)$ is a maximal DG ideal of $R$ and
$\pi'^{-1}(P)$ is the largest DG Poisson ideal contained in
$\pi'^{-1}(M)$. Since
\begin{center}
$R^{e}\pi'^{-1}(M)+R^{e}h_R(\pi'^{-1}(P))R^{e}\neq R^{e}$
\end{center}
by Lemma \ref{lem12}, there is a maximal left DG ideal $N$ of $R^{e}$
containing $R^{e}\pi'^{-1}(M)+R^{e}h_R(\pi'^{-1}(P))R^{e}$. Then
$X=R^{e}/N$ is a simple left DG $R^{e}$-module and so $X$ is a
simple left DG Poisson $R$-module.

Nota that $ann_{R^{e}}(X)$ is the largest DG ideal of $R^{e}$
contained in $N$:

since $ann_{R^{e}}(X)=\{u\in R^{e} | ~u\cdot X=0\}$, we have
\begin{center}
$0=u(u'+N)=uu'+N$
\end{center}
for all $u'+N\in X$, then $uu'\in N$. Set $u'=1$, we have $u\in N$,
that means $ann_{R^{e}}(X)\subseteq N$. If there exist a DG ideal
$B$ strictly contained in $N$ such that $ann_{R^{e}}(X)\subsetneq
B$, then there exist a $b\in B$ such that $b\notin ann_{R^{e}}(X)$,
we have $b(u_1+N)\neq 0$ for some $u_1+N\in X$, thus $bu_1\notin N$. But $b\in B, u_1\in R^{e}, B$ is a DG ideal, then $bu_1\in B\subseteq N$, which contradicts $bu_1\notin
N$. Hence $ann_{R^{e}}(X)$ is the largest DG ideal of $R^{e}$
contained in $N$.

Since $R^{e}\pi'^{-1}(P)$ is a DG ideal by Lemma \ref{lem8} and
$R^{e}\pi'^{-1}(P)\subseteq R^{e}\pi'^{-1}(M)\subseteq N$, we have
\begin{center}
$\pi'^{-1}(P)=(R^{e}\pi'^{-1}(P))\cap R\subseteq ann_{R^{e}}(X)\cap
R\subseteq N\cap R$
\end{center}
by Lemma \ref{lem8}. Since $\pi'^{-1}(M)\subseteq N\cap R\neq R$ and
$\pi'^{-1}(M)$ is a maximal DG ideal of $R$, we have
$\pi'^{-1}(M)=N\cap R$. Hence $\pi'^{-1}(P)=ann_{R^{e}}(X)\cap
R=ann_{R}(X)$ since $\pi'^{-1}(P)$ is the largest DG Poisson ideal
contained in $\pi'^{-1}(M)$.

Note that $ker(\eta)$ is a DG ideal of $R^{e}$ contained in $N$ by
Remark \ref{rmk1} since $I=ker(\pi')\subseteq \pi'^{-1}(P)$ and
$I\subseteq \pi'^{-1}(M)$. Hence $ker(\eta)X=0$ and so $X$ is a
simple left DG $A^e$-module with module structure induced by $\eta$.
Thus $X$ becomes a simple DG Poisson $A$-module.

Moreover, we have
\begin{equation}
\begin{split}
ann_A(X)&=m_A^{-1}(ann_{A^e}(X))=m_A^{-1}\eta(ann_{R^e}(X))\\
&=\pi'm_R^{-1}(ann_{R^e}(X))=\pi'(ann_{R}(X))=\pi'\pi'^{-1}(P)=P,
\end{split}\nonumber
\end{equation}
which completes the proof.
\end{proof}



\bigskip
\bibliographystyle{amsplain}

\begin{thebibliography}{10}
\bibitem{BG} K. A. Brown and I. Gordon, \emph{Poisson orders, symplectic reflection algebras and representation theory}, J. Reine Angew. Math. 559 (2003), 193--216.
\bibitem{CD} J. M. Casas and T. Datuashvili, \emph{Noncommutative Leibniz-Poisson algebras}, Comm. Algebra 34(7) (2006), 2507--2530.
\bibitem{CDL} J. M. Casas, T. Datuashvili and M. Ladra, \emph{Left-right noncommutative Poisson algebras}, Cent. Eur. J. Math. 12(1) (2014), 57--78.
\bibitem{CFL} A. S. Cattaneo, D. Fiorenza and R. Longoni, \emph{Graded Poisson algebras}, Encyclopedia Math. Phy. (2006), 560--567.
\bibitem{HL} T. J. Hodges and T. Levasseur, \emph{Primitive Ideals of $C_q[SL(3)]$}, Comm. Math. Phys. 156 (1993), 581--605.
\bibitem{HLe} T. J. Hodges and T. Levasseur, \emph{Primitive Ideals of $C_q[SL(n)]$}, J. Algebra 168 (1994), 455--468.
\bibitem{J} A. Joseph, \emph{Quantum Groups and their Primitive Ideals}, 3. Folge$\cdot$Band 29, A Series of Modern Surveys in Mathematics, Springer-Verlag, 1995.
\bibitem{KL} S.-J. Kang and K.-H. Lee, \emph{Gr\"{o}bner-Shirshov Bases for Irreducible $sl_{n+1}$-modules}, J. Algebra 232 (2000), 1--20.
\bibitem{KLe} S.-J. Kang and K.-H. Lee, \emph{Gr\"{o}bner-Shirshov Bases for Representation Theory}, J. Korean Math. Soc. 37 (2000), 55--72.
\bibitem{LWZ1} J.-F. L\"u, X. Wang and G. Zhuang, {\it Universal enveloping algebras of Poisson Hopf algebras}, J. Algebra 426 (2015), 92-136.
\bibitem{LWZ2} J.-F. L\"u, X. Wang and G. Zhuang, {\it Universal enveloping algebras of Poisson Ore-extensions}, Proc. Amer. Math. Soc. 143 (2015), 4633-4645.
\bibitem{LWZ} J.-F. L\"{u}, X. Wang and G. Zhuang, \emph{Universal
enveloping algebras of differential graded Poisson algeras}, Sci. China Math. 59 (2016), 849-860.
\bibitem{MPR} S. P. Mishchenko, V. M. Petrogradsky and A. Regev, \emph{Poisson PI algebras}, Trans. Amer. Math. Soc. 359(10) (2007), 4669--4694.
\bibitem{O} S.-Q. Oh, \emph{Poisson enveloping algebras}, Comm. Algebra 27 (1999), 2181--2186.
\bibitem{Oh} S.-Q. Oh, \emph{Symplectic Ideals of Poisson Algebras and the Poisson Structure Associated to Quantum Matrices}, Comm. Algebra 27 (1999), 2163--2180.
\bibitem{OPS} S.-Q. Oh, C.-G Park and Y.-Y Shin, \emph{A Poincar\'{e}-Birkhoff-Witt theorem for Poisson enveloping algebras}, Comm. Algebra 30(10) (2002), 4867--4887.
\bibitem{U} U. Umirbaev, \emph{Universal enveloping algebras and universal derivations of Poisson algebras}, J. Algebra 354 (2012), 77--94.
\bibitem{V} M. Van den Berger, \emph{Double Poisson algebras}, Trans. Amer. Math. Soc. 360(11) (2008), 5711--5769.
\bibitem{Va} M. Vancliff, \emph{Primitive and Poisson Spectra of Twists of Polynomial Rings}, Alg. Repre. Theory 2 (1999), 269--285.
\bibitem{X} P. Xu, \emph{Noncommutative Poisson algebras}, Amer. J. Math. 116(1) (1994), 101--125.
\bibitem{Xu} X.-P. Xu, \emph{Novikov-Poisson algebras}, J. Algebra 190(2) (1997), 253--279.
\bibitem{YYY} Y.-H. Yang, Y. Yao and Y. Ye, \emph{(Quasi-)Poisson enveloping algebras}, Acta Math. Sin. (Engl. Ser.) 29(1) (2013), 105--118.
\bibitem{YYZ} Y. Yao, Y. Ye and P. Zhang, \emph{Quiver Poisson algebras}, J. Algebra 312(2) (2007), 570--589.
\end{thebibliography}

\end{document}